\RequirePackage{fix-cm}
\documentclass[smallextended,draft]{svjour3} 
\smartqed 
\usepackage{graphicx}
\usepackage{mathptmx}      % use Times fonts if available on your TeX system
\usepackage{amsmath,amssymb}
\usepackage{bm}

\usepackage{Gmacros}

\journalname{INGE}
\begin{document}

\title{Wasserstein Riemannian Geometry of Gaussian Densities\thanks{The Authors wish to thank two anonymous referees for helpful comments. G.~Pistone acknowledges the support of de Castro Statistics and Collegio Carlo Alberto. He is a member of GNAMPA-INdAM.}}

% \titlerunning{Wasserstein Geometry of Gaussians}        % if too long for running head

\author{Luigi~Malag\`o \and
  Luigi~Montrucchio \and
  Giovanni~Pistone
}

\authorrunning{L.~Malag\`o, L.~Montrucchio, G.~Pistone} % if too long for running head

\institute{Luigi Malag\`o \at
Romanian Institute of Science and Technology - RIST, Str. Virgil Fulicea nr. 17 \\ 400022 Cluj-Napoca, Romania \\
\email{malago@rist.ro}
\and
Luigi Montrucchio \at
Collegio Carlo Alberto,
Piazza Vincenzo Arbarello 8,
10122 Torino, Italy \\
\email{luigi.montrucchio@unito.it}
\and
Giovanni Pistone \at
de Castro Statistics, Collegio Carlo Alberto, 
Piazza Vincenzo Arbarello 8, 
10122 Torino, Italy \\
\email{giovanni.pistone@carloalberto.org}
}

\date{REVISED \today}

\maketitle

\begin{abstract}
The Wasserstein distance on multivariate non-degenerate Gaussian densities is a Riemannian distance. After reviewing the properties of the distance and the metric geodesic, we present an explicit form of the Riemannian metrics on positive-definite matrices and compute its tensor form with respect to the trace inner product. The tensor is a matrix which is the solution to a Lyapunov equation. We compute the explicit formula for the Riemannian exponential, the normal coordinates charts and the Riemannian gradient. Finally, the Levi-Civita covariant derivative is computed in matrix form together with the differential equation for the parallel transport. While all computations are given in matrix form, nonetheless we discuss also the use of a special moving frame.
\end{abstract}

\section{Introduction\label{sec:introduction}}

Given two probability measures $\nu_1$ and $\nu_2$ on $\reals^n$, with finite second moments, consider the set $\mathcal P(\nu_1,\nu_2)$ of probability measures on the product sample space $\reals^{2n}$, such that the two $n$-dimensional margins have the prescribed distributions, $X_1 \sim \nu_1$ and $X_2 \sim \nu_2$. The index
\begin{equation*}\label{eq:firstG-general}
W^2 = \inf \setof{\expectat\mu {\normof{X_1-X_2}^2}}{\mu \in \mathcal P(\nu_1,\nu_2)}
\end{equation*}
as a measure of dissimilarity between distributions has been considered by many classical authors e.g., C. Gini, P. Levy, and M.R. Fr\'echet. There is considerable contemporary literature discussing the index $W$, which is usually called Wasserstein distance. E.g., the monograph by C.~Villani \cite{villani:2008optimal}. We want also to mention Y.~Brenier \cite{brenier:1991} and  R.J.~McCann \cite{mccann:2001}.

There is an important particular case, where the above problem reduces to the Monge transport problem. Borrowing the argument from M.~Knott and C.S.~Smith \cite{knott|smith:1984}, assume $\Phi \colon \reals^n \to \reals$ is a smooth strictly convex function and $\nabla \Phi(X_1) \sim \nu_2$. Clearly, the condition
\begin{equation*}
  \expectat \mu {\normof{X_1 - \nabla \Phi(X_1)}^2} \leq \expectat \mu { \normof{X_1-X_2}^2}, \quad \mu \in \mathcal P(\nu_1,\nu_2) \ ,
\end{equation*}
turns out to be equivalent to $\expectat \mu {X_1 \cdot \nabla \Phi(X_1)} \geq \expectat \mu {X_1 \cdot X_2}$. Latter inequality shows that the minimum quadratic distance is attained. In view of the new formulation, let us provide a proof.

If $\Psi$ denotes the convex conjugate of $\Phi$. We have
\begin{equation*}
  X_1 \cdot X_2 \leq \Phi(X_1) + \Psi(X_2)
\end{equation*}
and the equality case is
\begin{equation*}
  X_1 \cdot \nabla \Phi(X_1) = \Phi(X_1) + \Psi(\nabla \Phi X_1) \ .
\end{equation*}
By assumption $X_2 \sim \nabla \Phi (X_1)$ so that
\begin{multline*}
\expectat \mu {X_1 \cdot \nabla \Phi(X_1)} = \expectat \mu {\Phi(X_1) + \Psi(\nabla\Phi(X_1))} = \\ \expectat \mu {\Phi(X_1) + \Psi(X_2)} \geq \expectat \mu {X_1 \cdot X_2} \ .   
\end{multline*}

This argument, including an existence proof, is in Y. Brenier \cite{brenier:1991}. In the present paper we shall study the same problem where all the involved distributions are Gaussian. It would be feasible to reduce the Gaussian case to the general one. However, we resort to methods specially suited for this case.

\subsection{The Gaussian case}
\label{sec:gaussian-case}
Given two Gaussian distributions $\nu_i=\normalof n {\mu_i} {\Sigma_i}$, $i=1,2$, consider the set $\mathcal G(\nu_1,\nu_2)$ of Gaussian distributions on $\reals^{2n}$ such that the two $n$-dimensional margins have the prescribed distributions, $X_i \sim \nu_i$. The corresponding index is
\begin{equation}\label{eq:firstG}
W^2 = \inf \setof{\expectat\mu {\normof{X_1-X_2}^2}}{\mu \in \mathcal G(\nu_1,\nu_2)} \ .
\end{equation}

Observe that if $\mu_1=\mu_2=0$ and $U$ is a symmetric matrix such that $U\Sigma_1U = \Sigma_2$, then the previous argument applies by means of the convex function $\Phi(x) = \frac12 x^t U x$.%, which in turn provides an explicit value for $W^2$.

The value of $W^2$ in Eq.~\eqref{eq:firstG} as a function of the mean and the dispersion matrix has been computed by some authors, in particular: I. Olkin and F. Pukelsheim \cite{olkin|pukelsheim:1982}, D.~C. Dowson and B.~V. Landau \cite{dowson|landau:1982}, C.~R. Givens and R.~M. Shortt \cite{givens|shortt:1984}, M. Gelbrich \cite{gelbrich:1990}. They found the (equivalent) forms
\begin{equation} \label{eq:secondG}
\begin{aligned}
  W^2 & = \normof{\mu_1-\mu_2}^2 + \traceof{\Sigma_1+\Sigma_2 - 2 \left(\Sigma_1^{1/2}\Sigma_2\Sigma_1^{1/2}\right)^{1/2}} \\
& = \normof{\mu_1-\mu_2}^2 + \traceof{\Sigma_1+\Sigma_2 - 2 \left(\Sigma_1\Sigma_2\right)^{1/2}} \ .
\end{aligned}
\end{equation}

Further interpretations of $W$ are available. R. Bhatia et al. \cite{bhatia|jain|lim:2018} showed that $W$ is also the solution of constrained minimization problems for the Frobenius matrix norm $\normof M = \sqrt{\traceof{M^*M}}$, when $\mu_1=\mu_2=0$. Especially,
\begin{equation*}
W = \min \setof{\normof{\Sigma_1^{1/2}U-\Sigma_2^{1/2}V}}{\text{$U$ and $V$ orthogonal}} \ .
\end{equation*}
Notice that $\Sigma^{1/2}U$ is the generic transformation of the standard Gaussian to the Gaussian with dispersion matrix $\Sigma$.

Because of the exponent 2 in Eq.~\eqref{eq:firstG}, the $W$ distance is more precisely called $L^2$-Wasserstein distance. Other exponents or other distances could be used in the definition. The quadratic case is particularly relevant as $W$ is a Riemannian distance. More references will be given later.

In an Information Geometry perspective, we can mimic the argument of the seminal paper by Amari \cite{amari:1998natural}, who derived the notion of both Fisher metric and natural gradient, from the second order approximation of the Kullback-Leibler divergence. 

It will be shown (see Sec. 2) that the value $W^2$ of  Eq.~\eqref{eq:secondG} has the differential second-order expansion for small $H$:
\begin{equation}
  \label{eq:secondorderW}
  \traceof{\Sigma + (\Sigma+H) - 2\left(\Sigma^{1/2}(\Sigma+H)\Sigma^{1/2}\right)^{1/2}} \simeq  \traceof{\lyapunov \Sigma H \Sigma \lyapunov \Sigma H} \ ,
\end{equation}
where $\lyapunov \Sigma H = X$ is the solution to the Lyapunov equation $X \Sigma + \Sigma X = H$.

The quadratic form in the RHS of Eq.~\eqref{eq:secondorderW} provides a candidate to be the Riemannian inner product associated with the distance $W$. In addition, if $f$ is a smooth real function defined on a small $W$-sphere. i.e., $W(\Sigma,\Sigma+H) = \epsilon$ for small $\epsilon$, then the  increment $f(\Sigma+H)-f(\Sigma)$ is maximized along the direction
\begin{equation*}
  \label{eq:firstgrad}
  \Grad f(\Sigma) =\nabla f(\Sigma)\Sigma+\Sigma \nabla f(\Sigma) \ ,
\end{equation*}
where here $\nabla$ denotes the Euclidean gradient. The operator $\Grad$ is Amari's natural gradient, i.e., the Riemannian gradient.

It is remarkable that all geometric objects shown in the previous equations above may be expressed as matrix operations. In this paper, we proceed in developing systematically the Wasserstein geometry of Gaussian models according to such a formalism.

\subsection{Relations with the literature on the general transport theory}
\label{sec:relet-with-gener}
The Wasserstein distance and its relevant geometry can be studied non-parametrically also for general distributions. We do not pursue in this direction and refer to the monograph by C.~Villani \cite{villani:2008optimal}. The $L^2$-Wasserstein metric geometry has been shown to be Riemannian by F.~Otto \cite[\S 4]{otto:2001} and J.~Lott \cite{lott:2008calculations}. Cf. the earlier account by J.D. Lafferty \cite{lafferty:1988}.

Let us briefly discuss Otto's approach in the language of Information Geometry, i.e., with reference to S.~Amari and H.~Nagaoka \cite{amari|nagaoka:2000}. In view of  the non-parametric approach first introduced in \cite{pistone|sempi:95}, and denoted by $\mathcal M$ the set of $n$-dimensional Gaussian densities with zero mean, the vector bundle
\begin{equation*}
  \label{eq:1}
H\mathcal M = \setof{(\rho,\phi)}{\rho \in \mathcal M, \phi \in L^2(\rho), \int \phi \ \rho = 0}  
\end{equation*}
is the Amari Hilbert bundle on $\mathcal M$. The Hilbert bundle contains the statistical bundle whose fibers consist of the scores $\left. \derivby t \log \rho(t) \right|_{t=0}$ for all smooth curves $t \mapsto \rho(t) \in \mathcal M$ with $\rho(0)=\rho$. In turn, the statistical bundle is the tangent space of $\mathcal M$ considered as an exponential manifold, see \cite{pistone|sempi:95,pistone:2013GSI}.

In our present case, since the model $\mathcal M$ is an exponential family, the natural parameter is the concentration matrix $C = \Sigma^{-1}$. The log-likelihood is
\begin{equation*}
  \label{eq:2}
  \log \rho(y;C) = -\frac12 \log 2\pi + \frac12 \log \det C - \frac12 y^*C y \ .
\end{equation*}

If $V$ is a symmetric matrix, the derivative of $C \mapsto \log \rho(y;C)$ in the direction $V$ is
\begin{equation*}
  \label{eq:3}
  d_V \log \rho(y;C) = \frac12 \traceof{C^{-1}V} - \frac12 y^*Vy = \traceof{\phi(y;C)V}
\end{equation*}
where $\phi(y;C) = \frac12(C^{-1} - yy^*)$ is a symmetric matrix identified with a linear operator on symmetric matrices $\sym n$, equipped with the Frobenius inner product. The fiber at $\rho(\cdot;C)$ consists of the vector space of functions $\traceof{\phi(\cdot;C)V}$, $V \in \sym n$. The inner product in the Hilbert bundle, restricted to the parameterized statistical bundle, is the Fisher metric

\begin{multline}
  \label{eq:4}
  F_C(U,V) = \int d_U \log \rho(y;C)d_V \log \rho(y;C) \ \rho(y;C) \ dy = \\ - \int d_U \traceof{\phi(y;C)V} \ \rho(y;C) \ dx = \frac12 \traceof{UC^{-1}VC^{-1}} \ .\end{multline} 
The study of the Fisher metric in the Gaussian case has been done first by L.T.~Skovgaard \cite{skovgaard:1984}.

F. Otto \cite[\S 1.3]{otto:2001}, who was motivated by the study of a class of partial differential equation, considered a inner product defined on smooth functions of the $\rho$-fiber of the Hilbert bundle, as
\begin{equation}
  \label{eq:otto-metric}
(\phi_1,\phi_2) \mapsto \int \nabla \phi_1(x) \cdot \nabla \phi_2(x) \ \rho(x) \ dx \ .
\end{equation}
In the non-parametric case, Otto's metric of Eq.~\eqref{eq:otto-metric} is related to the Wasserstein distance, for a detailed study of such a metric see J.~Lott \cite{lott:2008calculations}. 

If we apply this definition to our score $\traceof{\phi(y;C)V} = \traceof{\frac12(C^{-1}-yy^*)V}$ and $V \in \sym n$, the gradient is $\nabla \traceof{\phi(y;C)V} = - V y$ and the metric becomes 

\begin{multline} \label{eq:metric-otto-exponential}
G_C(U,V) = \int \nabla \traceof{\phi(y;C)U} \cdot \nabla \traceof{\phi(y;C)V} \ \rho(y;C) \ dy = \\ \int y^* V U y \ \rho(y;\Sigma) \ dy = \traceof{UC^{-1}V} \ .    \end{multline}
The equivalence between the metric in Eq. \eqref{eq:metric-otto-exponential} and the one in Eq.~\eqref{eq:4} can be seen by a change of parameterization both in $\mathcal M$ and in each fiber. First, one must define the inner product at $\Sigma$ to be the inner product computed in the bijection $\Sigma \leftrightarrow C$, to get $\traceof{U \Sigma V}$, which is the form of the metric provided by A. Takatsu \cite[Prop. A]{takatsu:2011osaka}. Second, one has to change the parameterization on each fiber of the statistical bundle by $U \mapsto U \Sigma + \Sigma U$. The involved change of parameterization in the statistical bundle $(C,U) \mapsto (C^{-1},UC^{-1}+C^{-1}U)$ whose inverse is $(\Sigma,X) \mapsto (\Sigma^{-1},\lyapunov \Sigma X)$ produces the desired inner product. 

We mention also that the Machine Learning literature discusses a divergence introduced by A. Hyv\"arinen \cite{MR2249836}, which is related to Otto's metric. Precisely, in the concentration parameterization the Hyv\"arinen divergence is

\begin{multline*}
  \KH {D}{C} = \frac12 \int \aval{\nabla \log \rho(y;D) - \nabla \log \rho(y;C)}^2  \rho(y;C) \ dy = \\ \frac12 \int \aval{Dy-Cy}^2 \rho(y;C) \ dy = \traceof{C^{-1}(D-C)^2} \ ,
\end{multline*}
and the second derivative of $D \mapsto \KH D C$ at $C$ is
\begin{equation*}
  d^2\KH C C [X,Y] = \traceof{XC^{-1}Y} \ . 
\end{equation*}
In Statistics, Hyv\"arinen divergence is related to local proper scoring rules,  see M.~Parry et al. \cite{parry|dawid|lauritzen:2012}.

\subsection{Overview}\label{sec:overview}

The first two sections of the paper are mostly review of known material. In Sec.~\ref{sec:general-set-up} we recall some properties of the space of symmetric matrices. In particular: Riccati equation, Lyapunov equation, and the calculus regarding to the two mappings $\SQ \colon A \mapsto A^2$ and $\SQRT \colon A \mapsto A^{1/2}$. The mapping $\sigma \colon A \mapsto AA^*$, where $A$ is a non-singular square matrix is shown to be a submersion and the horizontal vectors at each point is computed. Despite of our manifold being finite dimensional, there is no need of choosing a basis, as all operations of interest are matrix operations. For that reason, we rely on the language of non-parametric differential geometry of W. Klingenberg \cite{klingenberg:1995} and S. Lang \cite{lang:1995}.

In Sec.~\ref{sec:quadr-diss-index} we discuss known results about the metric geometry induced by the Wasserstein distance. These results are re-stated in Prop. \ref{prop:CFT} and, for sake of completeness, we provide a further proof inspired by \cite{dowson|landau:1982}. It is possible to write down an explicit metric geodesic as done by R.J.~McCann \cite[Example 1.7]{mccann:2001}, see Prop.~\ref{prop:geo}.
The space of
non-degenerate Gaussian measures (or, equivalently, the space of positive
definite matrices) can be endowed with a Riemann structure that induces the Wasserstein distance. This is elaborated in Sec.~\ref{Sect:WASS}, where we use the presentation given by \cite{takatsu:2011osaka}, cf. also \cite{bhatia|jain|lim:2018}, which in turn adapts to the Gaussian case the original work \cite[\S 4]{otto:2001}. 

The remaining part of the paper is offered as a new contribution to this topic. The Wasserstein Riemannian metric turns out to be
\begin{equation}
  \label{eq:firstW}
  W_\Sigma(U,V) = \traceof{\lyapunov \Sigma U \Sigma\, \lyapunov \Sigma V} = \frac12 \traceof{\lyapunov \Sigma U V} \ ,
\end{equation}
at each matrix $\Sigma$, and where $U,V$ are symmetric matrices. By submersion methods  we study the more general problem of the horizontal surfaces in $\GLof n$, characterized in Prop.~\ref{prop:horizontal-surf}. As a specialized  case we get the Riemannian geodesic which agrees with the metric geodesic of Section 3. 

The explicit form of Riemannian exponential is obtained in Sec.~\ref{sec:riem-expon}. The natural (Riemannian) gradient is discussed in Sec.~\ref{sec:natural-gradient} and some applications to optimization are provided in Sec.~\ref{sec:optimization}. The analysis of the second-order geometry is treated in Sec.~\ref{sec:second-order}, where we compute the Levi-Civita covariant derivative, the Riemannian Hessian, and discuss other related topics. However, the curvature tensor will not be taken into consideration in the present paper.

In the final Sec.~\ref{sec:discussion}, we discuss the results in view of applications and in  Information Geometry of statistical sub-models of the Gaussian manifold.

\section{Symmetric matrices\label{sec:general-set-up}}

The set ${\mathcal{G}}^{n}$ of Gaussian distributions on ${\reals}^{n}$ is in 1-to-1 correspondence with the space of its parameters ${\mathcal{G}}^{n}\ni \operatorname{N}_{n}\left( \mu ,\Sigma \right) \leftrightarrow
(\mu ,\Sigma )\in {\reals}^{n}\times \operatorname{Sym}^{+}\left( n\right)$. Moreover, ${\mathcal{G}}^{n}$ is closed for the weak convergence and the
identification is continuous in both directions. A reference for Gaussian distributions is the monograph T.W.~Anderson \cite{anderson:2003}.

For ease of later reference, we recall a few results on spaces of matrices. General references are the monographs by P.~R. Halmos \cite{halmos:1958}, J.~R. Magnus and H. Neudecker \cite{magnus|neudecker:1999}, and R. Bhatia \cite{bhatia:2007PDM}.

The vector space of $n\times m$ real matrices is denoted by $\operatorname{M}(n\times m)$, while square matrices are denoted  $\operatorname{M}(n)=\operatorname{M}(n\times n)$. It is an Euclidean space of dimension $nm$ and the vectorization mapping $\operatorname{M}
(n\times m)\ni A\mapsto \mathbf{vec}\left( A\right) \in {\reals}^{nm}$
is an isometry for the Frobenius inner product $\left\langle A,B\right\rangle =(\mathbf{vec}\left( A\right) )^{\ast }(\mathbf{vec}\left( B\right) )=\operatorname{Tr}\left( AB^{\ast }\right)$.

Symmetric matrices $\sym n$ form a vector subspace of $M(n)$ whose orthogonal complement is the space of anti-symmetric matrices $\asym n$. We will find it convenient the use, with regard to symmetric matrices, of the equivalent inner product $
\left\langle A,B\right\rangle _{2}=\frac{1}{2}\operatorname{Tr}\left( AB\right)$, see e.g.  Eq.~\eqref{eq:onehalf} below. The closed pointed cone of non-negative-definite symmetric
matrices is denoted by $\psym n$ and its interior, the
open cone of the positive-definite
symmetric matrices, by $\ppsym n$. 

Given $A,B\in \operatorname{Sym}\left( n\right) $, the equation $TAT=B$ is called
Riccati equation. If $A\in \operatorname{Sym}^{++}\left( n\right) $ and $B\in 
\operatorname{Sym}^{+}\left( n\right) $, then the equation $TAT=B$ has unique
solution $T\in \operatorname{Sym}^{+}\left( n\right) $. In fact, from $TAT=B$ it follows $A^{1/2}TA^{1/2}A^{1/2}TA^{1/2}=A^{1/2}BA^{1/2}$ and, in turn, $A^{1/2}TA^{1/2}=\left( A^{1/2}BA^{1/2}\right) ^{1/2}$ because $T \in \psym n$. Hence, the solution to Riccati equation is
\begin{equation}
T=A^{-1/2}\left( A^{1/2}BA^{1/2}\right) ^{1/2}A^{-1/2} \ .  \label{eq:ricc}
\end{equation}
Notice that $\detof T = \detof A ^{-1/2} \detof B ^{1/2}$, consequently $\detof T > 0$ if $\detof B > 0$. In terms of random variables, if $X \in \normalof n 0 A$ and $Y = \normalof n 0 B$, then $T$ is the unique matrix of $\psym n$ such that $Y \sim TX$.

A more compact closed-form solution of the Riccati equation is available. Given $A \in \ppsym n$ and $B \in \psym n$, observe that $AB = A^{1/2}(A^{1/2}BA^{1/2})A^{-1/2}$. By similarity, the eigenvalues of $AB$ are non-negative, hence the square root
\begin{equation}
  \label{eq:sqrtAB}
  (AB)^{1/2} = A^{1/2}(A^{1/2}BA^{1/2})^{1/2}A^{-1/2} 
\end{equation}
is well defined, see \cite[Ex. 4.5.2]{bhatia:2007PDM}. Therefore, an equivalent formulation of  Eq.~\eqref{eq:ricc} is
\begin{equation}\label{eq:AB1}
T = A^{-1} A^{1/2}\left( A^{1/2}BA^{1/2}\right) ^{1/2}A^{-1/2} = A^{-1} (AB)^{1/2} \ .
\end{equation}
Since $AB = A(BA)A^{-1}$, the eigenvalues of $AB$ and $BA$ are identical, so that the same argument used before yields too 
\begin{equation}\label{eq:AB2}
T = (BA)^{1/2} A^{-1} \ .  
\end{equation}

The square mapping $\SQ \colon A \mapsto A^2$ is an injection of $\ppsym n$ onto itself with derivative $d_X \SQ(A) = XA + AX$. Hence, the derivative operator $d\SQ(A)$ is invertible. An alternative notation for the derivative we find convenient to use now and then is $d_X \SQ(A) = d \SQ(A)[X]$.

For each assigned matrix $V \in \sym n$, the matrix $X = (d \SQ(A))^{-1} V$ is the unique
solution $X$ in the space $\sym n$ to the Lyapunov
equation
\begin{equation}\label{eq:LIA}
V= X A + A X \ .  
\end{equation}
Its solution will be written $X = \lyapunov A V$. Clearly we have also
\begin{equation}\label{eq:obvious}
V =\lyapunov A V A + A \lyapunov A V \quad \text{and}
\quad X = \lyapunov A {XA+AX} \ .  
\end{equation}

The Lyapunov operator itself can be seen as a derivative. In fact, the inverse of the square mapping $\SQ$ is the square root mapping $\SQRT \colon \Sigma \to \Sigma^{1/2}$. By the derivative-of-the-inverse rule,
\begin{equation}\label{eq:DRO}
d_{V} \SQRT(\Sigma) = (d\SQ(\SQRT(\Sigma)))^{-1}[V] = \lyapunov {\Sigma^{1/2}} V \ .
\end{equation}

If $\Sigma$ is the dispersion of a non-singular Gaussian distribution, then $C = \Sigma^{-1} \in \ppsym n$ is the concentration matrix and represents an alternative and useful parameterization. From the Lyapunov equation $V = X\Sigma + \Sigma X$ we obtain $\Sigma^{-1}V\Sigma^{-1} = \Sigma^{-1}X + X\Sigma^{-1}$, hence 
\begin{equation*}\label{eq:ly1}
\lyapunov \Sigma V = \lyapunov {\Sigma^{-1}} {\Sigma^{-1}V\Sigma^{-1}} \quad \text{and} \quad \lyapunov {\Sigma^{-1}} U = \lyapunov \Sigma {\Sigma U \Sigma} \ .
\end{equation*}

Likewise, another useful formula is
\begin{equation} \label{eq:ly2}
  \lyapunov \Sigma V = \Sigma^{-1/2} \lyapunov \Sigma {\Sigma^{-1/2}V\Sigma^{-1/2}} \Sigma^{-1/2} \ .
\end{equation}

There is also a relation between the Lyapunov equation and the trace. From $X \Sigma + \Sigma X = V$, it follows $\Sigma^{-1} X \Sigma + X = \Sigma^{-1}V$. Then
\begin{equation}\label{eq:traceL}
\traceof{\lyapunov \Sigma V} = \frac12 \traceof{\Sigma^{-1}V} \ .  
\end{equation}

We will later need the derivative of the mapping $A \mapsto \lyapunov A V$, for a fixed $V$. Differentiating the first identity in Eq.~\eqref{eq:obvious} in the direction $U$, we have 
\begin{equation*}
0= d_U \lyapunov A V A + \lyapunov A V U + U \lyapunov A V + A \, d_U \lyapunov A V \ . 
\end{equation*}
Hence $d_{U}\lyapunov A V$ is the solution to the
Lyapunov equation 
\begin{equation*}
d_U \lyapunov A V A + A \ d_U \lyapunov A V =
- (\lyapunov A V U + U \lyapunov A V) \ ,
\end{equation*}
so that we get
\begin{equation}\label{eq:DEDE}
d_U \lyapunov A V  =
- \lyapunov A {\lyapunov A V U + U \lyapunov A V} \ .
\end{equation}

It will be useful in the following to evaluate the second derivative of the mapping $\SQRT \colon \Sigma \mapsto \Sigma^{1/2}$. From  Eqs.~\eqref{eq:DRO} and \eqref{eq:DEDE} it follows
\begin{equation*} \label{eq:DRO2}
  d^2 \SQRT(\Sigma)[U,V] = \lyapunov {\Sigma^{1/2}} {\lyapunov {\Sigma^{1/2}} V \lyapunov {\Sigma^{1/2}} U + \lyapunov {\Sigma^{1/2}}U \lyapunov {\Sigma^{1/2}} V} \ . 
\end{equation*}

Lyapunov equation plays a crucial role, as the linear operator $\mathcal L_A$ enters the expression of the Riemannian metric with respect to the standard inner product, see Eq.~\eqref{eq:firstW}.  As a consequence, the numerical implementation of the inner product $W_\Sigma(U,V)$ will require the computation of the matrix $\lyapunov \Sigma U$. There are many ways to write down the closed-form solution to Eq.~\eqref{eq:LIA}. They are discussed in \cite{bhatia:2007PDM}. However, efficient numerical solutions are not based on the closed forms, but rely on specialized numerical algorithms, as discussed by E.~L. Wachspress \cite{wachspress:2008} and by V. Simoncini \cite{simoncini:2016}.

We now turn to the computation of the second-order approximation of $W^2$ in Eq.~\eqref{eq:secondG}. 

Fix $\Sigma \in \ppsym n$ and let $H \in \sym n$ so that $(\Sigma \pm H) \in \ppsym n$. Hence, $\Sigma + \theta H\in \ppsym n$ for all $\theta \in [-1,+1]$. Consider the expression of $W^2$ with $\mu_1=\mu_2=0$, $\Sigma_1=\Sigma$, $\Sigma_2=\Sigma+\theta H$, namely
\begin{equation*}
 \theta \mapsto  W^2(\Sigma,\Sigma+\theta H) = 
  2 \traceof \Sigma + \theta \traceof H - 2 \traceof{\left(\Sigma^2+ \theta \Sigma^{1/2}H\Sigma^{1/2}\right)^{1/2}} \ .
\end{equation*}

By Eq.~\eqref{eq:DRO} and Eq.~\eqref{eq:traceL}, the first-order derivative is

\begin{multline*}
  \derivby \theta W^2(\Sigma,\Sigma+\theta H) = \traceof H - 2 \traceof{\lyapunov {\left(\Sigma^2+ \theta \Sigma^{1/2}H\Sigma^{1/2}\right)^{1/2}}{\Sigma^{1/2}H\Sigma^{1/2}}} = \\
  \traceof H - \traceof{\left(\Sigma^2+ \theta \Sigma^{1/2}H\Sigma^{1/2}\right)^{-1/2}\left(\Sigma^{1/2}H\Sigma^{1/2}\right)} \ .
\end{multline*}

Observe that $\left. \derivby \theta W^2(\Sigma,\Sigma+\theta H)\right|_{\theta=0} = 0$. 

The second derivative is 
\begin{equation*}
  \dderivby \theta W^2(\Sigma,\Sigma+\theta H) = \traceof{\derivby \theta \left(\Sigma^2+ \theta \Sigma^{1/2}H\Sigma^{1/2}\right)^{-1/2}\left(\Sigma^{1/2}H\Sigma^{1/2}\right)}
  \end{equation*}
  with
  
\begin{multline*}
\derivby \theta \left(\Sigma^2+ \theta \Sigma^{1/2}H\Sigma^{1/2}\right)^{-1/2} = 
\left(\Sigma^2+ \theta \Sigma^{1/2}H\Sigma^{1/2}\right)^{-1/2} \times \\ \lyapunov {\left(\Sigma^2+ \theta \Sigma^{1/2}H\Sigma^{1/2}\right)^{1/2}}{\Sigma^{1/2}H\Sigma^{1/2}}\left(\Sigma^2+ \theta \Sigma^{1/2}H\Sigma^{1/2}\right)^{-1/2} \ ,
\end{multline*}
so that

\begin{multline*}
  \left.  \dderivby \theta W^2(\Sigma,\Sigma+\theta H) \right|_{\theta=0} =
  \traceof{\Sigma^{-1}\lyapunov {\Sigma}{\Sigma^{1/2}H\Sigma^{1/2}}\Sigma^{-1}\Sigma^{1/2}H\Sigma^{1/2}} = \\
  \traceof{\Sigma^{-1/2}\lyapunov{\Sigma}{\Sigma^{1/2}H\Sigma^{1/2}}\Sigma^{-1/2}H} = \traceof{\lyapunov{\Sigma}{H}H} \ ,
\end{multline*}
where Eq.~\eqref{eq:ly2} has been used. Finally, observe that

\begin{multline}\label{eq:onehalf}
\traceof{\lyapunov{\Sigma}{H}\Sigma \lyapunov \Sigma H} = \traceof{ \lyapunov \Sigma H \lyapunov{\Sigma}{H}\Sigma} = \\ \frac12 \traceof{ \lyapunov \Sigma H \left(\lyapunov \Sigma H \Sigma + \Sigma \lyapunov \Sigma H\right)} = \frac12 \traceof{\lyapunov \Sigma H H} 
\end{multline}

We can conclude that
\begin{equation*}
W^2(\Sigma,\Sigma + \theta H) = \frac{\theta^2}2\traceof{\lyapunov{\Sigma}{H}H} + \smallo(\theta^2)  = \theta^2 \traceof{\lyapunov{\Sigma}{H} \Sigma \lyapunov \Sigma H} + \smallo(\theta^2) \ .  
\end{equation*}
Therefore, the bi-linear form in the RHS suggests the form of the Riemannian metric to be derived.

\subsection{The mapping $A \mapsto AA^*$} 
We study now the extension of the square operation to general invertible matrices, namely the mapping $\MMT : \GLof n \to \ppsym n$, defined by $\MMT(A) = AA^*$. Next proposition shows that this operation is a 
submersion. We recall first its definition, see \cite[Ch. 8, Ex. 8--10]{docarmo:1992} or \cite[\S II.2
]{lang:1995}.

Let ${\mathcal{O}}$ be an open set of the Hilbert space $H$, and $f\colon{
\mathcal{O}}\rightarrow{\mathcal{N}}$ a smooth surjection from the Hilbert
space $H$ onto a manifold ${\mathcal{N}}$, i.e., assume that for each $A\in{
\mathcal{O}}$ the derivative at $A$, $df(A)\colon H\rightarrow T_{f(A)}{
\mathcal{N}}$ is surjective. In such a case, for each $C\in {\mathcal{N}}$,
the fiber $f^{-1}(C)$ is a sub-manifold. Assigned a point $A\in f^{-1}(C)$, a
vector $U\in H$ is called vertical if it is tangent to the manifold $
f^{-1}(C)$. Each such a tangent vector $U$ is the velocity at $t=0$ of some
smooth curve $t\mapsto\gamma(t)$ with $\gamma(0)=A$ and $\dot{\gamma}(0)=U$.
Precisely, from $f(\gamma(t))=C$ for all $t$ we derive the characterization
of vertical vectors. We have $df(A)[\dot{\gamma}(0)]=0$ i.e., the tangent
space at $A$ is $T_{A}f^{-1}(f(A))=\operatorname{Ker}(df(A))$. The orthogonal space
to the tangent space $T_{A}f^{-1}(f(A))$ is called the space of horizontal vectors at $A$, 
\begin{equation*}
{\mathcal{H}}_{A}=\operatorname{Ker}(df(A))^{\perp}=\operatorname{Im}\left(
df(A)^{\ast}\right) \ .
\end{equation*}

Let us apply this argument to our specific case.
Let $\operatorname{GL}(n)\subset \operatorname{M}(n)$ be the open set of invertible matrices; $\operatorname{O}\left( n\right) $ the subgroup of $\operatorname{GL}(n)$ of orthogonal matrices; $\operatorname{Sym}^{\perp }\left( n\right) $ the
subspace of $\operatorname{M}(n)$ of anti-symmetric matrices. 
\begin{proposition}\label{prop:symcalculus}
\begin{enumerate}
\item For each given $A\in \operatorname{GL}(n)$ we have the orthogonal
splitting 
\begin{equation*}
\operatorname{M}(n)=\operatorname{Sym}\left( n\right) A\oplus \operatorname{Sym}^{\perp }\left(
n\right) (A^{\ast })^{-1}\ .
\end{equation*}
\item The mapping 
\begin{equation*}
\MMT \colon \operatorname{GL}(n)\ni A\mapsto AA^{\ast }\in \operatorname{Sym}^{++}\left(
n\right)
\end{equation*}
has derivative at $A$ given by $d_{X}\sigma (A)=XA^{\ast }+AX^{\ast }$. It is a submersion with fibers 
\begin{equation*}
\MMT^{-1}(C)=\setof{C^{1/2}R}{R\in\Oof n} \ .
\end{equation*}
\item The
kernel of the differential is 
\begin{equation*}
\operatorname{Ker}(d\sigma (A))=\operatorname{Sym}^{\perp }\left( n\right) (A^{\ast })^{-1}\ 
\end{equation*}
and its orthogonal complement, ${\mathcal{H}}_{A}=\operatorname{Ker}(d\sigma
(A))^{\perp },$ is 
\begin{equation*}
{\mathcal{H}}_{A}=\operatorname{Sym}\left( n\right) A.
\end{equation*}
\item The orthogonal projection of $X \in M(n)$ onto $\mathcal H_A$ is $\lyapunov {AA^*}{XA^*+AX^*}A$. 
 \end{enumerate}
\end{proposition}

\begin{proof} We provide here the proof for sake of completeness. See also \cite{takatsu:2011osaka} and \cite{bhatia|jain|lim:2018}. 
\begin{enumerate}
\item If $\left\langle B,CA\right\rangle =0$, for all $C\in\sym n$ i.e., $CA \in \psym n A$ , then $\operatorname{Tr}\left( BA^{\ast }C\right) =0$, so that $
BA^{\ast }\in \operatorname{Sym}^{\perp }\left( n\right) $ that is, $B \in \asym n (A^*)^{-1}$.
\item Let the matrix $A$ be an element in the fiber manifold $\MMT^{-1}(AA^{\ast })$. The derivative of $\MMT$ at $A$, $X \mapsto XA^{\ast }+AX^{\ast }$, is surjective, because for each $
W\in \operatorname{Sym}\left( n\right) $ we have $d\MMT(A)\left[ \frac{1}{2}
W(A^{\ast })^{-1}\right] =W$. Hence $\MMT$ is a submersion and the fiber $\MMT^{-1}(AA^*)=\setof{(AA^*)^{1/2}R}{R \in \Oof n}$ is a sub-manifold of $\operatorname{GL}(n)$. 

\item Let us compute the splitting of $\Mof n$ into the kernel of $d\MMT(A)$ and its orthogonal: $\operatorname{M}(n)=\operatorname{Ker}(d\sigma (A))\oplus {\mathcal{H}}_{A}$. The vector space tangent to $\MMT^{-1}(AA^{\ast })$ at $A$ is the
kernel of the derivative at $A$: 
\begin{equation*}
\kerof{d\sigma (A)} =\left\{ X\in \operatorname{M}(n)|\text{ }XA^{\ast
}+AX^{\ast }=0\right\} = \left\{ X\in \operatorname{M}(n)|\text{ }(AX^{\ast })^{\ast }=-AX^{\ast }\right\} \ .
\end{equation*}
Therefore, $X\in \operatorname{Ker}(d\sigma (A))$ if, and only if, $AX^{\ast }\in \operatorname{
Sym}^{\perp }\left( n\right) $, i.e., $\operatorname{Ker}(d\sigma (A))=\operatorname{Sym}
^{\perp }\left( n\right) (A^{\ast })^{-1}$. We have just proved that this implies $\mathcal{H}_{A}= \sym n A$.
\item Consider the decomposition of $X$ into the horizontal and the vertical part:
$X = C A + D (A^*)^{-1}$ with $C \in \sym n$ and $D \in \asym n$.
By transposition, we get $X^* = A^* C - A^{-1} D$. From the previous two equations, we obtain the two equations $XA^* = C (AA^*) + D$ and $AX^* = (AA^*)C - D$. The sum of the two previous equations is $XA^* + AX^* = C(AA^*)+(AA^*)C$, which is a Lyapunov equation having solution $C = \lyapunov {AA^*} {XA^* + AX^*}$. It follows that the projection is $CA = \lyapunov {AA^*} {XA^* + AX^*}A$
\end{enumerate}
\end{proof}

\section{Wasserstein distance}\label{sec:quadr-diss-index}

The aim of this section is to discuss the Wasserstein distance for the Gaussian case as well as the equation for the associated metric geodesic. Most of its content is an exposition of known results.

\subsection{Block-Gaussian} 
Let us suppose that the dispersion matrix $\Sigma \in \psym {2n}$
is partitioned into $n\times n$ blocks, and consider random variables $X$ and $Y$ such that
\begin{equation*}
\begin{bmatrix}
X \\ 
Y
\end{bmatrix}
\sim \operatorname{N}_{2n}\left( \mu ,\Sigma \right) ,\quad \Sigma =
\begin{bmatrix}
\Sigma _{1} & K \\ 
K^{\ast } & \Sigma _{2}
\end{bmatrix}
\ ,
\end{equation*}
so that $K_{ij}=\operatorname{Cov}\left( X_{i},Y_{j}\right) $ if $i=1,\dots ,n$ and $
j=(n+1),\dots ,2n$. It follows that $K_{ij}^{2} \leq (\Sigma
_{1})_{ii}(\Sigma _{2})_{jj} \leq \frac12 \left((\Sigma
_{1})_{ii} + (\Sigma _{2})_{jj}\right)$, which in turn imply the bounds 
\begin{equation}\label{eq:ubo}
\left\Vert K\right\Vert _{2}^{2}\leq \operatorname{Tr}\left( \Sigma _{1}\right) 
\operatorname{Tr}\left( \Sigma _{2}\right)  \quad \text{and} \quad \sup_{ij}\left\vert K_{ij}\right\vert \leq \frac{1}{2}(\operatorname{Tr}\left(
\Sigma _{1}\right) +\operatorname{Tr}\left( \Sigma _{2}\right) ) \ .
\end{equation}

For mean vectors $\mu _{1},\mu _{2}\in {\reals}^{2}$ and
dispersion matrices $\Sigma _{1},\Sigma _{2}\in \operatorname{Sym}^{+}\left(
n\right) $, define the set of jointly Gaussian distributions with given
marginals to be 
\begin{equation*}
{\mathcal{G}}((\mu _{1},\Sigma _{1}),(\mu _{2},\Sigma _{2}))=\left\{ \operatorname{N}
_{2n}\left( 
\begin{bmatrix}
\mu _{1} \\ 
\mu _{2}
\end{bmatrix}
,
\begin{bmatrix}
\Sigma _{1} & K \\ 
K^{\ast } & \Sigma _{2}
\end{bmatrix}
\right) \right\} \ ,
\end{equation*}
and the Gini dissimilarity index  

\begin{multline} \label{eq:Gi}
W^{2}((\mu _{1},\Sigma _{1}),(\mu _{2},\Sigma _{2})) = \\
\inf\setof{\mathbb{E}\left[ \left\Vert X-Y\right\Vert ^{2}\right]}{\begin{bmatrix}
X \\ 
Y
\end{bmatrix}
\sim \gamma , \gamma \in {\mathcal{G}}((\mu _{1},\Sigma _{1}),(\mu
_{2},\Sigma _{2}))}
= \\ \left\Vert \mu _{1}-\mu _{2}\right\Vert ^{2}+\operatorname{Tr}\left( \Sigma
_{1}\right) +\operatorname{Tr}\left( \Sigma _{2}\right) -2\sup \setof{ \operatorname{Tr}
\left( K\right) }{
\begin{bmatrix}
\Sigma _{1} & K \\ 
K^{\ast } & \Sigma _{2}
\end{bmatrix}
\in \operatorname{Sym}^{+}\left( 2n\right)}  
\end{multline}
Actually, in view of either of the bounds in Eq.~\eqref{eq:ubo}, the set ${
\mathcal{G}}((\mu _{1},\Sigma _{1}),(\mu _{2},\Sigma _{2}))$ is compact and the $\inf$ is attained.

It is easy to verify that
\begin{equation*}
W((\mu_{1},\Sigma_{1}),(\mu_{2},\Sigma_{2}))=\sqrt{\min\setof{ \mathbb{E}
\left[ \left\Vert X-Y\right\Vert ^{2}\right] }{
\begin{bmatrix}
X \\ 
Y
\end{bmatrix}
\sim\gamma,\gamma\in{\mathcal{G}}((\mu_{1},\Sigma_{1}),(\mu_{2},\Sigma
_{2}))} }  \label{eq:D}
\end{equation*}
defines a distance on the space ${\mathcal{G}}_{n}\simeq{\reals}^{n}\times \operatorname{
Sym}^{+}\left( n\right) $. The symmetry of $W$ is clear as well as the triangle inequality, by considering Gaussian distributions on ${\reals}
^{n}\times{\reals}^{n}\times{\reals}^{n}$ with given marginals. To
conclude, assume that the $\min$ is reached at some $\overline{\gamma}$. Then 
\begin{equation*}
0=W((\mu_{1},\Sigma_{1}),(\mu_{2},\Sigma_{2}))=\mathbb{E}_{\overline{\gamma}}
\left[ \left\vert X-Y\right\vert ^{2}\right] \Leftrightarrow\mu_{1}=\mu
_{2}\quad{\text{and}}\quad\Sigma_{1}=\Sigma_{2}\ .
\end{equation*}

A further observation is that distance $W$ is homogeneous i.e., 
\begin{equation*}
W((\lambda\mu_{1},\lambda^{2}\Sigma_{1}),(\lambda\mu_{2},\lambda^{2}\Sigma
_{2}))=\lambda W((\mu_{1},\Sigma_{1}),(\mu_{2},\Sigma_{2})),\quad\lambda
\geq0\ .
\end{equation*}

\subsection{Computing the quadratic dissimilarity index}

We will present a proof as given by Dowson and Landau \cite
{dowson|landau:1982}, but with some corrections.

Given $\Sigma _{1},\Sigma _{2}\in \operatorname{Sym}^{+}\left( n\right) $, each admissible $K$'s in \eqref{eq:Gi} belongs to a compact set of $\operatorname{M}
(n) $ thanks to bound \eqref{eq:ubo}, so the maximum of the function $2
\operatorname{Tr}\left( K\right) $ is reached. Therefore, we are led to study the problem
\begin{equation} \label{eq:JJJ}
\left\{ \begin {aligned} \alpha &(\Sigma _1,\Sigma _2)=\max _{K\in
\operatorname {M}(n)}2\operatorname {Tr}\left (K\right )\\ &{\text {subject
to}}\\ &\Sigma = \begin {bmatrix} \Sigma _1&K\\ K^*&\Sigma _2\end {bmatrix}
\in \operatorname {Sym}^+\left (2n\right )\end {aligned}\right.
\end{equation}
The value of the similar problem with $\max$ replaced by $\min $ will be denoted by $\beta (\Sigma _{1},\Sigma _{2}).$

\begin{proposition}\ 
\label{prop:uno}
\begin{enumerate}
\item Let $\Sigma_{1},\Sigma_{2}\in\operatorname{Sym}^{+}\left( n\right) $. Then
\begin{equation*}
\alpha(\Sigma_{1},\Sigma_{2})=2\operatorname{Tr}\left( \left( \Sigma
_{1}^{1/2}\Sigma_{2}\Sigma_{1}^{1/2}\right) ^{1/2}\right) \text{ and }
\beta(\Sigma_{1},\Sigma_{2})=-\alpha(\Sigma_{1},\Sigma_{2})\ .
\end{equation*}
\item If moreover $\detof {\Sigma_1} > 0$, then
\begin{equation*}
\alpha(\Sigma_{1},\Sigma_{2}) = 2 \traceof{(\Sigma_1\Sigma_2)^{1/2}} \ .
\end{equation*}
\end{enumerate}
\end{proposition}

\begin{proof}[point $(1)$]\ 
A symmetric matrix $\Sigma\in\operatorname{Sym}\left( 2n\right) $ is non-negative
defined if, and only if, it is of the form $\Sigma=SS^{\ast}$, with $S\in
\operatorname{M}\left( 2n\right) $. Given the block structure of $
\Sigma $ in \eqref{eq:JJJ}, we can write 
\begin{equation*}
\begin{bmatrix}
\Sigma_{1} & K \\ 
K^{\ast} & \Sigma_{2}
\end{bmatrix}
=
\begin{bmatrix}
A \\ 
B
\end{bmatrix}
\begin{bmatrix}
A^{\ast} & B^{\ast}
\end{bmatrix}
=
\begin{bmatrix}
AA^{\ast} & AB^{\ast} \\ 
BA^{\ast} & BB^{\ast}
\end{bmatrix}
,
\end{equation*}
where $A$ and $B$ are two matrices in $\operatorname{M}(n\times2n).$

Therefore, problem \eqref{eq:JJJ} becomes 
\begin{equation*}
\left\{ \begin {aligned} \alpha &(\Sigma _1,\Sigma _2)=\max _{A,B\in
\operatorname {M}(n\times 2n)}2\operatorname {Tr}\left (AB^*\right )\\
&{\text {subject to}}\\ &\Sigma _1=AA^*,\quad \Sigma _2=BB^*\end {aligned}
\right.  \label{eq:HKP}
\end{equation*}

We have already observed that the optimum exists, so the necessary
conditions of Lagrange theorem allows us to characterize this optimum.
However, the two constraints $\Sigma_{1}=AA^{\ast}$ and $\Sigma_{2}=BB^{
\ast} $ are not necessarily regular at every point (i.e., the Jacobian of
the transformation may fail to be of full rank at some point), so we must
take into account that the optimum could be an irregular point. To this
purpose, as a customary, we shall adopt Fritz John first-order formulation
for the Lagrangian (see \cite{mangasarian|fromovitz:1967}).

We shall initially assume that both $\Sigma_{1}$ and $\Sigma_{2}$ are
non-singular.

Let then $\left( \nu_{0},\Lambda,\Gamma\right) \in\left\{ 0,1\right\} \times
\operatorname{Sym}\left( n\right) \times\operatorname{Sym}\left( n\right) $, $\left(
\nu_{0},\Lambda,\Gamma\right) \neq\left( 0,0,0\right) $, where the symmetric
matrices $\Lambda$ and $\Gamma$ are the Lagrange multipliers. The Lagrangian
function will be 
\begin{align*}
L & =2\nu_{0}\operatorname{Tr}\left( AB^{\ast}\right) -\operatorname{Tr}\left( \Lambda
AA^{\ast}\right) -\operatorname{Tr}\left( \Gamma BB^{\ast }\right) \\
& =2\nu_{0}\operatorname{Tr}\left( AB^{\ast}\right) -\operatorname{Tr}\left( A^{\ast}\Lambda
A\right) -\operatorname{Tr}\left( B^{\ast}\Gamma B\right)
\end{align*}
The first-order conditions of $L$ lead to
\begin{equation} \label{eq:FIR}
\left\{ \begin {aligned} &\nu _{0}B=\Lambda A,\quad \nu _{0}A=\Gamma B\\
&\Sigma _1=AA^*,\quad \Sigma _2=BB^*\end {aligned}\right. \ .  
\end{equation}

In the case $\nu_{0}=1,$ i.e., the case of stationary regular points, Eq.~\eqref{eq:FIR} becomes 
\begin{equation} \label{eq:BBB}
\left\{ \begin {aligned} &B=\Lambda A,\quad A=\Gamma B\\ &\Sigma
_1=AA^*,\quad \Sigma _2=BB^*\end {aligned}\right. \ , 
\end{equation}
which in turn implies

\begin{equation} \label{eq:YUP}
\left\{ \begin {aligned} \Lambda \Sigma _{1}\Lambda &=\Sigma _{2}\\ \Gamma
\Sigma _{2}\Gamma &=\Sigma _{1}\end {aligned}\right. ,\quad\Lambda,\Gamma\in
\operatorname{Sym}\left( n\right)  
\end{equation}
and further 
\begin{equation*}
K=\Sigma_{1}\Lambda=\Gamma\Sigma_{2}.
\end{equation*}

Of course, Eqs.~\eqref{eq:YUP} could be more general than Eqs.~\eqref{eq:BBB} and thus possibly contain undesirable solutions. In this light, we
establish the following facts, in which both matrices $\Sigma _{1}$ and $
\Sigma _{2}$ must be nonsingular. Notice that in this case Eqs.~\eqref{eq:YUP} imply that both $\Lambda $ and $\Gamma $ are nonsingular as well.

\bigskip
\noindent\emph{Claim 1: If $(\Gamma,\Lambda)$ is a solution to \eqref{eq:YUP} and $\Lambda^{-1}=\Gamma$, then the couple $(\Gamma,\Lambda)$ are Lagrange
multipliers of Problem \eqref{eq:JJJ}.}

Actually, let $\Sigma _{1}=AA^{\ast }$, $A\in \operatorname{M}(n\times 2n)$ be any
representation of the matrix $\Sigma _{1}$. Define $B=\Lambda A$ so that $
A=\Lambda ^{-1}B=\Gamma B$. Moreover 
\begin{equation*}
BB^{\ast }=\Lambda AA^{\ast }\Lambda =\Lambda \Sigma _{1}\Lambda =\Sigma
_{2}\ ,
\end{equation*}
and so $\left( \Lambda ,\Gamma \right) $ are multipliers associated with the
feasible point $(A,B)$.

\bigskip
\noindent\emph{Claim 2: The set of solutions to \eqref{eq:YUP}, such that $
\Gamma^{-1}=\Lambda$, is not empty. In particular, there is a unique pair $
\left( \widetilde{\Lambda},\widetilde{\Gamma}\right) $ where both $
\widetilde{\Lambda}$ and $\widetilde{\Gamma}$ are positive definite}.

We have already observed that Eqs.~\eqref{eq:YUP} imply that $\Lambda$ and $
\Gamma$ are nonsingular. Moreover, we have $\Gamma^{-1}\Sigma_{1}\Gamma
^{-1}=\Sigma_{2}$. Recalling that Riccati's equation has one and only one
solution in the class of positive definite matrices, then $X=\Lambda
=\Gamma^{-1}$.

Now we proceed to study the solutions to $\Lambda\Sigma_{1}\Lambda=
\Sigma_{2} $ and we shall show that Eq \eqref{eq:YUP} has infinitely many
solutions. In correspondence to each one $\Lambda$, the value of the
objective function will be given by $2\operatorname{Tr}\left( K\right) =2\operatorname{Tr}
\left( \Sigma_{1}\Lambda\right) $. Therefore, we must select the matrix $
\Lambda$ such that ${\mathrm{Tr}}\left( \Sigma_{1}\Lambda\right) $ be
maximized.

Following \cite{dowson|landau:1982}, we define 
\begin{equation*}
R=\Sigma_{1}^{1/2}\Lambda\Sigma_{1}^{1/2}\in\operatorname{Sym}\left( n\right) \ ,
\end{equation*}
so that, in view of \eqref{eq:YUP}, we have 
\begin{equation} \label{eq:tre}
R^{2}=\Sigma_{1}^{1/2}\Lambda\Sigma_{1}^{1/2}\Sigma_{1}^{1/2}\Lambda\Sigma
_{1}^{1/2}=\Sigma_{1}^{1/2}\Lambda\Sigma_{1}\Lambda\Sigma_{1}^{1/2}=\Sigma
_{1}^{1/2}\Sigma_{2}\Sigma_{1}^{1/2}\in\operatorname{Sym}^{+}\left( n\right) \ .
\end{equation}

Moreover, 
\begin{equation*}
\operatorname{Tr}\left( R\right) =\operatorname{Tr}\left(
\Sigma_{1}^{1/2}\Lambda\Sigma_{1}^{1/2}\right) =\operatorname{Tr}\left( \Sigma
_{1}^{1/2}\Sigma_{1}^{1/2}\Lambda\right) =\operatorname{Tr}\left( \Sigma
_{1}\Lambda\right) =\operatorname{Tr}\left( K\right) \ .
\end{equation*}

Eq. \eqref{eq:tre} shows that, though the Lagrangian can have many rest
points (i.e., many solutions $\Lambda$) the matrix $R^{2}=\Sigma_{1}^{1/2}
\Sigma _{2}\Sigma_{1}^{1/2}\in\operatorname{Sym}^{+}\left( n\right) $ remains
constant. Not so the value of the objective function $\operatorname{Tr}\left(
K\right) =\operatorname{Tr}\left( R\right) $ which depends on $R$ (i.e., on $\Lambda$
).

Let 
\begin{equation*}
R^{2}=\sum_{k}\lambda _{k}E_{k}
\end{equation*}
denote the spectral decomposition of $R^{2}$, then the solutions to $R$ will
be 
\begin{equation*}
R=\sum_{k}\varepsilon _{k}\lambda _{k}^{1/2}E_{k}
\end{equation*}
with $\varepsilon _{k}=\pm 1$. Hence $\operatorname{Tr}\left( K\right) =\operatorname{Tr}
\left( R\right) $ will be maximized whenever $\varepsilon _{k}\equiv 1$ and
so $R\in \operatorname{Sym}^{+}\left( n\right) $. Clearly the objective function
will be minimized if $\varepsilon _{k}\equiv -1$. From now on the proof of
the $\min $ statement follows similarly.

Hence the maximum of the trace occurs at 
\begin{equation*}
R=\left( \Sigma_{1}^{1/2}\Sigma_{2}\Sigma_{1}^{1/2}\right) ^{1/2}\ ,
\end{equation*}
namely $\Lambda=\Sigma_{1}^{-1/2}\left( \Sigma_{1}^{1/2}\Sigma_{2}\Sigma
_{1}^{1/2}\right) ^{1/2}\Sigma_{1}^{-1/2}.$ Thanks to Claims 1-2 this matrix
is a multiplier of the Lagrangian and so we would have 
\begin{equation} \label{eq:DAT}
\alpha\left( \Sigma_{1},\Sigma_{2}\right) =2{\mathrm{Tr}}\left( \Sigma
_{1}^{1/2}\Sigma_{2}\Sigma_{1}^{1/2}\right) ^{1/2},  
\end{equation} 
as long as the optimum is attained at a regular point. In fact, to complete
the proof, we must still examine the case $\nu_{0}=0$, for which Eq. \eqref{eq:FIR} becomes 
\begin{equation*}
\Lambda A=0,\quad\Gamma B=0\ .
\end{equation*}
It follows 
\begin{align*}
\Lambda\Sigma_{1} & =\Lambda AA^{\ast}=0 \\
\Gamma\Sigma_{2} & =\Gamma BB^{\ast}=0\ ,
\end{align*}
and consequently $\Lambda=\Gamma=0$. Therefore there is no irregular point,
provided $\Sigma_{1}$ and $\Sigma_{2}$ are not singular matrices. So we have
proved the relation \eqref{eq:DAT} under the above assumptions.

Last step will be that of extending our result to possibly singular matrices $\Sigma
_{1} $ and $\Sigma _{2}$.

Given the two matrices $\Sigma_{1},\Sigma_{2}\in\operatorname{Sym}^{+}\left(
n\right) $, set 
\begin{equation*}
\Sigma_{1}\left( \varepsilon\right) =\Sigma_{1}+\varepsilon I_{n}{\text{ \
and }}\Sigma_{2}\left( \varepsilon\right) =\Sigma_{2}+\varepsilon I_{n}{
\text{, \ with }}\varepsilon\in\lbrack0,1]\ .
\end{equation*}
If $\varepsilon>0$, then 
\begin{equation*}
\operatorname{det}\left( \Sigma_{i}+\varepsilon I\right)
=\prod_{j=1}^{n}(\lambda_{i,j}+\varepsilon)>0,\quad i=1,2\ .
\end{equation*}
where $\lambda_{i,j}$, $j=1,\dots,n$ is a set of eigenvalues of $\Sigma_{i}$
, $i=1,2$. Let us consider the parametric programming problem 
\begin{equation*}
\left\{ \begin {aligned} \alpha &(\Sigma _1(\varepsilon ),\Sigma
_2(\varepsilon ))=\max _{K\in \operatorname {M}(n)}2\operatorname {Tr}\left
(K\right )\\ &{\text {subject to}}\\ & \begin {bmatrix} \Sigma
_1(\varepsilon )&K\\ K^*&\Sigma _2(\varepsilon )\end {bmatrix} \in
\operatorname {Sym}^+\left (2n\right )\end {aligned}\right.
\end{equation*}
Observe that the feasible region is contained in a compact set independent
of $\varepsilon\in\left[ 0,1\right] $ because of the bound \eqref{eq:ubo}.

Now the continuity of the optimal value $\varepsilon \mapsto \alpha (\Sigma
_{1}(\varepsilon ),\Sigma _{2}(\varepsilon ))$ follows easily from Berge
maximum theorem, see for instance \cite[Th. 17.31]{aliprantis|border:2006}.
Hence 
\begin{equation*}
\alpha (\Sigma _{1},\Sigma _{2})=\lim_{\varepsilon \rightarrow 0}\alpha
(\Sigma _{1}(\varepsilon ),\Sigma _{2}(\varepsilon ))=2\operatorname{Tr}\left(
(\Sigma _{1}^{1/2}\Sigma _{2}\Sigma _{1}^{1/2})^{1/2}\right)
\end{equation*}
and the assertion is proved for any $\Sigma _{1},\Sigma _{2}\in \operatorname{Sym}
^{+}\left( n\right) $.
\end{proof}

\begin{proof}[point $(2)$]
 From Eq.~\eqref{eq:sqrtAB} we have 

\begin{equation*}
\traceof{\left(\Sigma_1^{1/2} \Sigma_2 \Sigma_1^{1/2}\right)^{1/2}} =  \traceof{\Sigma_1^{1/2}\left(\Sigma_1^{1/2} \Sigma_2 \Sigma_1^{1/2}\right)^{1/2}\Sigma_1^{-1/2}} = \traceof{\left(\Sigma_1\Sigma_2\right)^{1/2}} \ .  
\end{equation*}

\end{proof}

The following result provides exact both lower and upper
bounds of $\mathbb{E}\left[ \left\Vert X-Y\right\Vert ^{2}\right] $.

\begin{proposition}
\label{prop:CFT}Let $X,Y$ be multivariate Gaussian random variables taking values in ${\reals}
^{n}$ and having means $\mu_{1}$ and $\mu_{2}$ and dispersion matrices $
\Sigma_{1}$ and $\Sigma_{2}$ respectively. Then 

\begin{multline*}
\left\Vert \mu_{1}-\mu_{2}\right\Vert ^{2}+\operatorname{Tr}\left( \Sigma
_{1}+\Sigma_{2}-2\left( \Sigma_{1}^{1/2}\Sigma_{2}\Sigma_{1}^{1/2}\right)
^{1/2}\right) \leq\mathbb{E}\left[ \left\Vert X-Y\right\Vert ^{2}\right] \leq
\\
\left\Vert \mu_{1}-\mu_{2}\right\Vert ^{2}+\operatorname{Tr}\left( \Sigma
_{1}+\Sigma_{2}+2\left( \Sigma_{1}^{1/2}\Sigma_{2}\Sigma_{1}^{1/2}\right)
^{1/2}\right) \ .
\end{multline*}
If $\det\Sigma_{1}\neq0$, then the extremal values are attained at the joint
distribution of 

\begin{multline*}
\begin{bmatrix}
X \\ 
\mu_{2}\pm T(X-\mu_{1})
\end{bmatrix}
\sim \\ \operatorname{N}_{2n}\left( 
\begin{bmatrix}
\mu_{1} \\ 
\mu_{2}
\end{bmatrix}
,
\begin{bmatrix}
\Sigma_{1} & \pm T\Sigma_{1} \\ 
\pm\Sigma_{1}T & \Sigma_{2}
\end{bmatrix}
\right) = \operatorname{N}_{2n}\left( 
\begin{bmatrix}
\mu_{1} \\ 
\mu_{2}
\end{bmatrix}
,
\begin{bmatrix}
\Sigma_{1} & \pm (\Sigma_2\Sigma_{1})^{1/2} \\ 
\pm (\Sigma_1\Sigma_2)^{1/2} & \Sigma_{2}
\end{bmatrix}
\right) \ ,
\end{multline*}
respectively, where $T\in\operatorname{Sym}^{+}\left( n\right) $ is the solution to
the Riccati equation $T\Sigma_{1}T=\Sigma_{2}$.
\end{proposition}

\begin{proof} From Proposition \ref{prop:uno} and Eq.~\eqref{eq:Gi}, it
follows 
\begin{equation*}
\begin{aligned}
\min \left[ \left\Vert X-Y\right\Vert ^{2}\right] & =\left\Vert \mu _{1}-\mu
_{2}\right\Vert ^{2}+\operatorname{Tr}\left( \Sigma _{1}\right) +\operatorname{Tr}\left(
\Sigma _{2}\right) -2\operatorname{Tr}\left( \left( \Sigma _{1}^{1/2}\Sigma
_{2}\Sigma _{1}^{1/2}\right) ^{1/2}\right) \ , \\
\max \left[ \left\Vert X-Y\right\Vert ^{2}\right] & =\left\Vert \mu _{1}-\mu
_{2}\right\Vert ^{2}+\operatorname{Tr}\left( \Sigma _{1}\right) +\operatorname{Tr}\left(
\Sigma _{2}\right) +2\operatorname{Tr}\left( \left( \Sigma _{1}^{1/2}\Sigma
_{2}\Sigma _{1}^{1/2}\right) ^{1/2}\right) \ .
\end{aligned}
\end{equation*}

To check the extremal points it suffices to observe that, in view of
relation \eqref{eq:ricc}: 
\begin{equation*}
\operatorname{Tr}\left( T\Sigma _{1}\right) =\operatorname{Tr}\left( \Sigma
_{1}^{-1/2}\left( \Sigma _{1}^{1/2}\Sigma _{2}\Sigma _{1}^{1/2}\right)
^{1/2}\Sigma _{1}^{1/2}\right) =\operatorname{Tr}\left( \left( \Sigma
_{1}^{1/2}\Sigma _{2}\Sigma _{1}^{1/2}\right) ^{1/2}\right) .
\end{equation*}
Hence it is verified that the extremal values are attained at $Y=\mu _{2}\pm
T(X-\mu _{1})$. In the second form of the distribution we are using Eq.~\eqref{eq:AB1} and Eq.~\eqref{eq:AB2}.
\end{proof}

The $W$-distance defines on ${\reals} \times \ppsym n$ a metric geometry with geodesics. This result is due to \cite{mccann:2001}.

\begin{proposition}
\label{prop:geo}
The relation 
\begin{equation} \label{eq:GINI}
W\left( (\mu_{1},\Sigma_{1}),(\mu_{2},\Sigma_{2})\right) =\sqrt{\left\Vert
\mu_{1}-\mu_{2}\right\Vert ^{2}+\operatorname{Tr}\left( \Sigma_{1}+\Sigma
_{2}-2\left( \Sigma_{1}^{1/2}\Sigma_{2}\Sigma_{1}^{1/2}\right) ^{1/2}\right) 
}  
\end{equation}
defines a distance on ${\reals}^{n}\times\operatorname{Sym}^{+}\left( n\right) $. The geodesic from $(\mu _{1},\Sigma _{1})$ to $(\mu _{2},\Sigma _{2})$, with 
$(\mu _{1},\Sigma _{1}),(\mu _{2},\Sigma _{2})\in {\reals}^{n}\times \operatorname{Sym}^{++}\left( n\right)$, is the curve
\begin{equation*}\label{eq:GGG}
\Gamma \colon [0,1] \ni t \mapsto \left(\mu(t),\Sigma(t)\right) \ ,  
\end{equation*}
where $\mu(t) = (1-t)\mu _{1} + t\mu_{2}$ and

\begin{multline*}
\Sigma(t) = ((1-t)I+tT)\Sigma _{1}((1-t)I+tT) = \\ (1-t)^2 \Sigma_1 + t^2 \Sigma_2 + t(1-t)\left((\Sigma_1\Sigma_2)^{1/2} + (\Sigma_2\Sigma_1)^{1/2}\right) \ ,
\end{multline*}
and $T$ is the (unique) non-negative definite solution to the Riccati equation $T\Sigma _{1}T=\Sigma _{2}$. 
\end{proposition}

\begin{proof}
Clearly, $\Gamma (0)=\left( \mu _{1},\Sigma
_{1}\right) $ and $\Gamma (1)=\left( \mu _{2},\Sigma _{2}\right) $. Let us
compute the distance between $\Gamma (0)$ and the point 
\begin{equation*}
\Gamma (t)=(\mu (t),\Sigma (t))=\left( \mu _{1}+t(\mu _{2}-\mu
_{1}),((1-t)I+tT)\Sigma _{1}((1-t)I+tT)\right) .
\end{equation*}

We have 
\begin{equation*}
\begin{aligned}
\Sigma_{1}^{1/2}\Sigma(t)\Sigma_{1}^{1/2} &
=\Sigma_{1}^{1/2}((1-t)I+tT)\Sigma_{1}((1-t)I+tT)\Sigma_{1}^{1/2} \\
& =\left( \Sigma_{1}^{1/2}((1-t)I+tT)\Sigma_{1}^{1/2}\right) \left(
\Sigma_{1}^{1/2}((1-t)I+tT)\Sigma_{1}^{1/2}\right) \ ,
\end{aligned}
\end{equation*}
so that 
\begin{equation*}
\left( \Sigma_{1}^{1/2}\Sigma(t)\Sigma_{1}^{1/2}\right) ^{1/2}=\Sigma
_{1}^{1/2}((1-t)I+tT)\Sigma_{1}^{1/2}\ ,
\end{equation*}
and hence 
\begin{multline*}
\operatorname{Tr}\left( \left( \Sigma_{1}^{1/2}\Sigma(t)\Sigma_{1}^{1/2}\right)
^{1/2}\right) = \\ \operatorname{Tr}\left(
\Sigma_{1}^{1/2}((1-t)I+tT)\Sigma_{1}^{1/2}\right) =(1-t)\operatorname{Tr}\left(
\Sigma_{1}\right) +t\operatorname{Tr}\left( T\Sigma_{1}\right) \ .
\end{multline*}

We have
\begin{equation*}
\begin{aligned}
\operatorname{Tr}\left( \Sigma(t)\right) & =\operatorname{Tr}\left(
((1-t)I+tT)\Sigma_{1}((1-t)I+tT)\right) \\
& =(1-t)^{2}\operatorname{Tr}\left( \Sigma_{1}\right) +2t(1-t)\operatorname{Tr}\left(
T\Sigma_{1}\right) +t^{2}\operatorname{Tr}\left( \Sigma_{2}\right)
\end{aligned}
\end{equation*}

Collecting all the above results, 

\begin{multline*}
\operatorname{Tr}\left( \Sigma_{1}+\Sigma(t)-2\left(
\Sigma_{1}^{1/2}\Sigma(t)\Sigma_{1}^{1/2}\right) ^{1/2}\right) =\operatorname{Tr}
\left( \Sigma_{1}\right) + \\
(1-t)^{2}\operatorname{Tr}\left( \Sigma_{1}\right) +2t(1-t)\operatorname{Tr}\left(
T\Sigma_{1}\right) +t^{2}\operatorname{Tr}\left( \Sigma_{2}\right) -2(1-t)\operatorname{Tr}
\left( \Sigma_{1}\right) -2t\operatorname{Tr}\left( T\Sigma_{1}\right) = \\
t^{2}\operatorname{Tr}\left( \Sigma_{1}\right) +t^{2}\operatorname{Tr}\left(
\Sigma_{2}\right) -2t^{2}\operatorname{Tr}\left( T\Sigma_{1}\right) =t^{2}\operatorname{Tr}
\left( \Sigma_{1}+\Sigma_{2}-2\left(
\Sigma_{1}^{1/2}\Sigma_{2}\Sigma_{1}^{1/2}\right) ^{1/2}\right) \ .
\end{multline*}

In conclusion,

\begin{multline*}
W(\Gamma(0),\Gamma(t)) = \\
\sqrt{\left\Vert \mu(0)-\mu(t)\right\Vert ^{2}+\operatorname{Tr}\left(
\Sigma(0)+\Sigma(t)-2\left( \Sigma(0)^{1/2}\Sigma(t)\Sigma(0)^{1/2}\right)
^{1/2}\right) } = \\ tW(\Gamma(0),\Gamma(1))\ .
\end{multline*}
\end{proof}

We end this section by adding a few remarks.

In metric space, the definition of geodesic we use here is related to Merger convexity property, see \cite[p. 78]{papadopoulos:2014}. A stronger definition requires
the proportionality of the distance between couple of points on the curve,
i.e., 
\begin{equation*}
W\left( \Gamma (s),\Gamma (t)\right) =\left\vert t-s\right\vert W\left(
\Gamma (0),\Gamma (1)\right) ,
\end{equation*}
for $s,t\in \left[ 0,1\right] $. It will be proved later that in fact our
geodesics enjoy such a stronger property.

Clearly Proposition \ref{prop:geo} still
holds under the only assumption that $\Sigma_{1}$ is not singular, but
the case in which both the distributions are degenerate remains excluded.
 
The simplest example occurs when the two subspaces, $\operatorname{Range}\Sigma_{1}$
and $\operatorname{Range}\Sigma_{2}$, are orthogonal. In this case, for all joint distribution of the random vector $(X,Y),$ with marginals $
X\sim\operatorname{N}_{2}\left( 0,\Sigma_{1}\right) $ and $Y\sim \operatorname{N}_{2}\left(
0,\Sigma_{2}\right) ,$ the values of $X$ and $Y$ will lie into orthogonal
subspaces, so that $XY^{\ast}=0.$ Hence $\left\Vert X-Y\right\Vert
^{2}=\left\Vert X\right\Vert ^{2}+\left\Vert Y\right\Vert ^{2}$, and 
\begin{equation*}
\mathbb{E}\left\Vert X-Y\right\Vert ^{2}=\mathbb{E}\left\Vert X\right\Vert
^{2}+\mathbb{E}\left\Vert Y\right\Vert ^{2}=\operatorname{Tr}\left(
\Sigma_{1}\right) +\operatorname{Tr}\left( \Sigma_{2}\right) .
\end{equation*}
So any joint distribution $(X,Y)$ attains the optimal value $\sqrt {\operatorname{Tr}
\left( \Sigma_{1}\right) +\operatorname{Tr}\left( \Sigma_{2}\right) }.$

If we now define $X(t)=(1-t)X+tY$, then 
\begin{equation*}
\mathbb{E}\left[ \left\Vert X-X(t)\right\Vert ^{2}\right] =\mathbb{E}\left[
t^{2}\left\Vert X-Y\right\Vert ^{2}\right] =t^{2}\left[ \operatorname{Tr}\left(
\Sigma _{1}\right) +\operatorname{Tr}\left( \Sigma _{2}\right) \right] ,
\end{equation*}
consequently $X(t)$ is the geodesic joining the two random vectors $X$ and $Y$.

The previous example can be extended by taking two singular matrices 
\begin{equation*}
\Sigma _{1}=\sigma _{1}^{2}vv^{\ast }\text{ and }\Sigma _{2}=\sigma
_{2}^{2}ww^{\ast }
\end{equation*}
where $v\neq w\in \reals^{n}$ and $\left\Vert v\right\Vert =\left\Vert
w\right\Vert =1$. Clearly, $\operatorname{Range}\Sigma _{1}\cap \operatorname{Range}
\Sigma _{2}=\left\{ 0\right\} $ and they are one-dimensional spaces spanned
by vectors $v$ and $w$, respectively (it is not restrictive to assume $
v^{\ast }w\geq 0$, too). By Eq.~\eqref{eq:GINI}, 
\begin{equation*}
G\left( \Sigma _{1},\Sigma _{2}\right) =\sqrt{\sigma _{1}^{2}+\sigma
_{2}^{2}-2\sigma _{1}\sigma _{2}v^{\ast }w}.
\end{equation*}
Despite singularity of these matrices, it can be directly found the point
realizing the minimum in \eqref{eq:Gi}, which is the singular matrix
in $\operatorname{Sym}^{+}\left( 2n\right) $: 
\begin{equation*}
\left[ 
\begin{array}{cc}
\sigma _{1}^{2}vv^{\ast } & \sigma _{1}\sigma _{2}vw^{\ast } \\ 
\sigma _{1}\sigma _{2}wv^{\ast } & \sigma _{2}^{2}ww^{\ast }
\end{array}
\right] =\left[ 
\begin{array}{c}
\sigma _{1}v \\ 
\sigma _{2}w
\end{array}
\right] \left[ 
\begin{array}{cc}
\sigma _{1}v^{\ast } & \sigma _{2}w^{\ast }
\end{array}
\right] .
\end{equation*}

\section{Wasserstein Riemannian geometry}
\label{Sect:WASS}

We have seen how to compute the geodesic for the distance $W$. Since the component $\reals^n$ carries the standard Euclidean geometry, we focus on the geometry of the matrix part, i.e., we
shall restrict our analysis to 0-mean distributions $\operatorname{N}_{n}\left(0,\Sigma \right)$. Moreover, $\Sigma$ will be assumed to be positive definite. Our purpose is to endow the open set $\ppsym n$ with a structure of Riemannian manifold whose metric tensor generates the Wasserstein distance. The Riemannian metric is obtained by pushing forward the Euclidean geometry of square matrices to the space of dispersion matrices via the mapping $\MMT \colon A \mapsto AA^* = \Sigma$. This approach has been introduced by F.~Otto \cite{otto:2001} in the general non-parametric case and developed in the Gaussian case by A.~Takatsu \cite{takatsu:2011osaka} and R.~Bhatia \cite{bhatia|jain|lim:2018}.

In view of Prop.~\ref{prop:symcalculus}, $\MMT\colon\operatorname{GL}
(n)\rightarrow\operatorname{Sym}^{++}\left( n\right) \subset\operatorname{M}(n)$ is a submersion and  ${\mathcal{H}}_{A}=\operatorname{Sym}\left( n\right) A$ is the space of
horizontal vectors at $A$.

We recall that a submersion $f\colon \operatorname{GL}(n)\rightarrow \operatorname{Sym}
^{++}\left( n\right) $ is called Riemannian if for all $A$ the differential restricted to horizontal vectors
\begin{equation*}
\left.df(A)\right|_{{\mathcal{H}}_{A}}\colon {\mathcal{H}}_{A}\rightarrow T_{f(A)}\operatorname{Sym}^{++}\left(
n\right) = \sym n
\end{equation*}
is an isometry i.e., 
\begin{equation}\label{eq:RiemannianSubmersion}
U,V\in {\mathcal{H}}_{A}\Rightarrow \left\langle
df(A)[U],df(A)[V]\right\rangle _{f(A)}=\left\langle U,V\right\rangle \ .
\end{equation}

A linear isometry is always 1-to-1 and, if it is onto, we can write backward that 
\begin{equation*}
X,Y\in T_{f(A)}\operatorname{Sym}^{++}\left( n\right) \Rightarrow \left\langle
X,Y\right\rangle _{f(A)}=\left\langle \left( \left. df(A)\right\vert _{{
\mathcal{H}}_{A}}\right) ^{-1}X,\left( \left. df(A)\right\vert _{{\mathcal{H}
}_{A}}\right) ^{-1}Y\right\rangle \ .
\end{equation*}
Conversely, the previous equation provides the definition of a metric on $\ppsym n$ for which the submersion $f$ is Riemannian.

If $U_A$ is the projection of $U$ on $\mathcal H_A$, then $df(A)[U] = df(A)[U_A]$ and Eq.~\eqref{eq:RiemannianSubmersion} becomes
\begin{multline*}
U,V\in\sym n\Rightarrow \left\langle
df(A)[U],df(A)[V]\right\rangle _{f(A)} = \\ \scalarat{f(A)}{df(A)[U_A]}{df(A)[V_A]} = \left\langle U_A,V_A\right\rangle \ .
\end{multline*}

In general, a submersion induces a local diffeomorphisms from horizontal spaces to the image manifold. In our case, the submersion $\sigma$ provides a global parameterization of the manifold of symmetric matrices. Fix a matrix $A\in \operatorname{GL}(n)$ such that $\MMT
(A)=AA^{\ast }=\Sigma $, and consider the open convex cone 
\begin{equation*}
{\mathcal{H}}_{A}^{++}=\operatorname{Sym}^{++}\left( n\right) A\subset {\mathcal{H}}
_{A}.  \label{eq:KKK}
\end{equation*}
We denote by $\MMT _{A}$ the restriction to ${\mathcal{H}}_{A}^{++}$ of $
\MMT $.

\begin{proposition}
For all $A \in \GLof n$, the mapping 
\begin{equation*}
\MMT_{A}\colon{\mathcal{H}}_{A}^{++}\ni B\mapsto BB^{\ast}=C\in \operatorname{Sym}
^{++}\left( n\right)
\end{equation*}
is a surjective bijection, with inverse 
\begin{equation*}
\MMT_{A}^{-1}(C) = C^{-1/2}(C^{1/2}\Sigma
C^{1/2})^{1/2}C^{-1/2}A\ .
\end{equation*}
\end{proposition}

\begin{proof}
For each $C\in \operatorname{Sym}^{++}\left( n\right) $,
the equation 
\begin{equation*}
C=BB^{\ast }=(BA^{-1}A)(BA^{-1}A)^{\ast }=(BA^{-1})\Sigma (BA^{-1})^{\ast }
\end{equation*}
is a Riccati equation for $BA^{-1}$. As $B \in \operatorname{
Sym}^{++}\left( n\right) A$, we have $BA^{-1}\in \operatorname{Sym}^{++}\left(
n\right) $ and 
\begin{equation*}
BA^{-1}=C^{-1/2}(C^{1/2}\Sigma C^{1/2})^{1/2}C^{-1/2}
\end{equation*}
is the unique solution.
\end{proof}

We come now to the point, i.e., the construction of a metric based on horizontal vectors at a given matrix $\Sigma$. We are here using  Prop.~\ref{prop:symcalculus}.

\begin{proposition}
\label{prop:NOV} The inner product 
\begin{equation*}\label{eq:NOV}
\left\langle U,V\right\rangle _{\Sigma }\equiv W_{\Sigma }(U,V)= \traceof{\lyapunov \Sigma U \Sigma \lyapunov \Sigma V},\quad
U,V\in \operatorname{Sym}\left( n\right) ,
\end{equation*}
defines a metric on $\ppsym n$ such that $\MMT \colon A\mapsto AA^{\ast}$ is a Riemannian submersion.
\end{proposition}

\begin{proof}
Let $X\in \operatorname{M}(n)$ and consider the decomposition of $X=X_{V}+X_{H}$
with $X_{V}$ vertical at $A$ and $X_{H}$ horizontal at $A$. Then $d\MMT
(A)[X]=d\MMT (A)[X_{H}]$ and the restriction of the derivative $d\MMT
(A) $ to the vector space ${\mathcal{H}}_{A}$ of horizontal vectors at $A$
is 1-to-1 onto the tangent space of $\operatorname{Sym}^{++}\left( n\right) $ at $
AA^{\ast }$, that is, $\operatorname{Sym}\left( n\right) $. For such a restriction, for each $H\in {\mathcal{H}}_{A},$ 

\begin{multline*}
\left. U=d\MMT (A)[H]=HA^{\ast }+AH^{\ast }=HA^{-1}AA^{\ast
}+A(HA^{-1}A)^{\ast }\right. \\
\left. =(HA^{-1})AA^{\ast }+AA^{\ast }(HA^{-1})^{\ast }=(HA^{-1})AA^{\ast
}+AA^{\ast }(HA^{-1})\ ,\right.
\end{multline*}
so that the inverse mapping of the restriction is given by 
\begin{equation} \label{eq:ASR}
H=\left( \left. d\MMT (A)\right\vert _{{\mathcal{H}}_{A}}\right) ^{-1}(U)=
\mathcal{L}_{AA^{\ast }}[U]A\ ,  
\end{equation}
Let us push-forward the inner product from ${\mathcal{H}}_{A}$ to $
T_{AA^{\ast}}\operatorname{Sym}^{++}\left( n\right) $.

From Eq.~\eqref{eq:ASR}, we have

\begin{multline*}
\left. W_{AA^{\ast }}(U,V)=\left\langle \left( \left. d\MMT (A)\right\vert
_{{\mathcal{H}}_{A}}\right) ^{-1}(U),\left( \left. d\MMT (A)\right\vert _{{
\mathcal{H}}_{A}}\right) ^{-1}(V)\right\rangle =\right. \\
\left. \left\langle \mathcal{L}_{AA^{\ast }}[U]A,\mathcal{L}_{AA^{\ast
}}[V]A\right\rangle =\operatorname{Tr}\left( \mathcal{L}_{AA^{\ast }}[U]AA^{\ast }
\mathcal{L}_{AA^{\ast }}[V]\right) .\right.
\end{multline*}
which depends on $AA^{\ast }=\Sigma $ only.
\end{proof}

Next proposition provides a useful tensorial form of Wasserstein Riemannian metric.

\begin{proposition}\label{prop:altWmetric}
  It holds 
\begin{equation*}
W_{\Sigma }(U,V) = \frac{1}{2}\left\langle \mathcal{L}_{\Sigma }[U],V\right\rangle \equiv
\left\langle \mathcal{L}_{\Sigma }[U],V\right\rangle _{2}.
\end{equation*}
\end{proposition}

\begin{proof}
We have 
\begin{equation*}
\operatorname{Tr}\left( \mathcal{L}_{\Sigma}[U]\Sigma \mathcal{L}_{\Sigma}
[V]\right) =\operatorname{Tr}\left( \mathcal{L}_{\Sigma }[V]\Sigma \mathcal{L}
_{\Sigma }[U]\right) =\operatorname{Tr}\left( \mathcal{L}_{\Sigma }[U]\mathcal{L}
_{\Sigma }[V]\Sigma \right) \ ,
\end{equation*}
and, taking the semi-sum of the first and the last term of the previous equation,
\begin{equation*}
W_{\Sigma }(U,V) = \frac{1}{2}\operatorname{Tr}\left\{ \mathcal{L}_{\Sigma }[U]
\left[ \mathcal{L}_{\Sigma }[V]\Sigma +\Sigma \mathcal{L}_{\Sigma }[V]\right]
\right\} = \frac{1}{2}\operatorname{Tr}\left\{ \mathcal{L}_{\Sigma }[U]V\right\}  \ .
\end{equation*}
\end{proof}

After having shown in Prop.~\ref{prop:geo} the existence of a metric geodesic for the Wasserstein distance,
connecting a pair of matrices $\Sigma _{1},\Sigma _{2} \in \ppsym n$, we prove that the same curve is the Wasserstein
Riemannian geodesic, see R.J.~McCann  \cite{mccann:1997} and also \cite{takatsu:2011osaka,bhatia|jain|lim:2018}. More generally, we now discuss the existence of affine horizontal surfaces in $\GLof n$ and the existence of geodesically convex surfaces in $\ppsym n$. As a particular case, the result gives rise to the desired Riemannian geodesics.

A surface $\theta \mapsto A(\theta) \in\GLof n$, with $\theta \in \Theta$ and $\Theta$ open subset of $\reals^n$, is called horizontal for the submersion $\sigma \colon A \mapsto AA^*$, if $\partial/\partial{\theta_j} A(\theta) \in \mathcal H_{A(\theta)}$ for each $j$ and $\theta$, i.e.,
\begin{equation}\label{eq:geodesicx}
\left(\partiald {\theta_j} A(\theta)\right) A(\theta)^{-1} \in \sym n \ .
\end{equation}
A surface is horizontal if, and only if, every smooth curve which lies in it is horizontal.

\begin{proposition}\label{prop:horizontal-surf}
\begin{enumerate}
\item The surface $\Theta \ni \theta \mapsto A(\theta) \in \GLof n$ is horizontal for $\sigma$ if, and only if, 
\begin{equation}\label{eq:geodesicy}
\partiald {\theta_j} A^*(\theta) A(\theta)=A^*(\theta) \partiald {\theta_j} A(\theta), \quad j=1,\dots,k\ , \quad \theta \in \Theta \ .
\end{equation}
\item Let 
\begin{equation}\label{eq:geodesicz}
A(\theta) = A_0 + \sum_{i=1}^k \theta_i (A_i - A_0) \ , \quad \theta \in \Theta \ ,
\end{equation}
be a surface in $\GLof n$ with the $k$-simplex of $\reals^k$ contained in $\Theta$. The surface is horizontal if, and only if,
\begin{equation*}\label{eq:geodesicS}
A^*_j A_i = A^*_i A_j \ , \quad i,j=0,\dots,k \ .
\end{equation*}
\item \label{prop:horizontal-surf3} Let be given $\Sigma_0,\Sigma_1 \in \ppsym n$ and choose $A_0,A_1$ such that  $\Sigma_0 = A_0A_0^*$ and $\Sigma_1 = A_1A_1^*$. The line
\begin{equation}\label{eq:horizontal-line}
A(\theta) = (1-\theta) A_0+\theta A_1 
\end{equation}
is horizontal for $\theta$ in an open interval containing 0 and 1 if, and only if, $A_1 = TA_0$ with $T \in \ppsym n$. This implies $T$ is the solution of the Riccati equation $T\Sigma_0T = \Sigma_1$.
\item \label{prop:horizontal-surf4} Let be given $\Sigma_j = A_jA_j^*\in \ppsym n$, $j=0,1\dots,k$. The surface 
\begin{equation*}
\theta \mapsto A_0 + \sum_{j=0}^k \theta_k (A_j-A_0) 
\end{equation*}
is horizontal in an open set of parameters containing the $k$-simplex if, and only if, $A_i = T_{ij} A_j$ with $T_{ij} \in \ppsym n$, $i,j=0,\dots,k$.

\end{enumerate}
\end{proposition} 

\begin{proof}
\begin{enumerate}
\item Eq.~\eqref{eq:geodesicx} is equivalent to $A^*(\theta)^{-1}\partial / \partial {\theta_j} A^*(\theta) = \partial / \partial {\theta_j} A(\theta)A(\theta)^{-1}$ hence to $\partial / \partial {\theta_j} A^*(\theta) A(\theta) = A^*(\theta) \partial / \partial {\theta_j} A(\theta)$.
\item For the surface in Eq.~\eqref{eq:geodesicz} we have $\partial/\partial {\theta_j} A(\theta) = A_j$ so that Eq.~\eqref{eq:geodesicy} becomes
\begin{equation*}
A_j^*(\theta) A(\theta)=A^*(\theta) A_j(\theta), \quad j=1,\dots,k\ , \quad \theta \in \Theta \ .
\end{equation*}
If $\theta = 0$, it holds $A_j^*A_0 = A_0^*A_j$, $j = 1,\dots,k$. If $\theta=e_i$ then it holds $A_j^* A_i = A_i A_j^*$ for $i,j=1,\dots,k$. The converse holds by linearity. 
\item Assume $\theta \mapsto A(\theta)$ of Eq.~\eqref{eq:horizontal-line} is horizontal on $\Theta$. Then, from the previous item we know $A_1^*A_0 = A_0^* A_1$. In turn, this implies ${A_0^*}^{-1}A_1^* = A_1 A_0^{-1}$, hence $T = A_1A_0^{-1} \in \sym n$. It follows $T \Sigma_0 T = A_1A_0^{-1} \Sigma_0 (A_0^*)^{-1} A_1^* = \Sigma_1$. It remains to show that $T$ is positive definite. Actually, it holds
\begin{equation*}
(1 - \theta) A_0 + \theta A_1 = \left((1 - \theta) I + \theta T\right)A_0 \in \GLof n \ , \quad \theta \in \Theta \ .
\end{equation*}
If $\lambda_i$ are eigenvalues of the matrix $T$, then the eigenvalues of the matrix $(1 - \theta) I + \theta T$ are $(1-\theta) + \theta \lambda_i$. As they are never zero for any $\theta\in [0,1]$, it follows that no $\lambda_i$ can be negative. The $\lambda_i$ are not zero by assumption and the conclusion $T \in \ppsym n$ follows. 

Conversely, if $T \in \ppsym n$ and $TA_0  = A_1$, then $A_1^* A_0 = A_0^* T A_0$ is symmetric. Consequently, for all $\theta$ such that $(1-\theta)A_0+\theta A_1 \in \GLof n$ the curve is horizontal. On the other hand, $(1-\theta)I + \theta T$ is the convex combination of positive definite matrices then it is positive definite on an open interval containing $[0,1]$.

\item The proof follows exactly the same arguments as in the 2-points case of the previous item.
\end{enumerate}
\end{proof}

We conclude by discussing the existence of the geodetic surfaces that have been characterized in the previous proposition. The result shows that there is equality between the metric geodesic derived from the Wasserstein distance and the the geodesic we obtain from the submersion argument. Moreover, we characterize the existence of geodesically convex surfaces with given vertices.

\begin{corollary}
\begin{enumerate}
\item 
Given $\Sigma_0, \Sigma_1 \in \ppsym n$, there exists an open interval $\Theta \supset [0,1]$ such that the curve
\begin{equation}\label{eq:geodesic-new}
\Sigma(\theta) = \left((1-\theta)I + \theta T\right) \Sigma_0 \left((1-\theta)I + \theta T\right) \ , \quad \theta \in \Theta \ ,
\end{equation}
is the Wasserstein Riemannian geodesic through $\Sigma_0$ and $\Sigma_1$, with $T\Sigma_0T = \Sigma_1$.
\item Let $\Sigma_0, \dots,\Sigma_k \in \ppsym n$, there exists an open set $\Theta$ containing the $k$-simplex such that the surface
\begin{equation*}\label{eq:geodesic-neww}
\Sigma(\theta) = \left(I + \sum_{j=1}^k\theta (T_j-I)\right) \Sigma_0 \left(I + \sum_{j=1}^k\theta (T_j-I)\right) \ , \quad \theta \in \Theta \ ,
\end{equation*}
is the Wasserstein Riemannian geodesic surface through $\Sigma_0, \dots,\Sigma_k$ if, and only if, the matrices $T_j$, which are the positive definite solution of the Riccati equations $T_j \Sigma_0T_j$, $j=1,\dots,k$, pairwise commute.
\end{enumerate}
\end{corollary}

\begin{proof}
\begin{enumerate}
\item Pick $A_0 = \Sigma_0^{1/2} U$, with $U \in \Oof n$, and $A_1 = TA_0$, where $T$ is the positive definite solution of the Riccati equation $T\Sigma_0T = \Sigma_1$ and so $A_1A_1^* = \Sigma_1$. By Prop.~\ref{prop:horizontal-surf}, Item~\ref{prop:horizontal-surf3}, $\theta \mapsto A(\theta)$ is horizontal in $\GLof n$.  Consequently, $\Sigma(\theta)=A(\theta)A^*(\theta)$ is a geodesic.
\item In view of Prop.~\ref{prop:horizontal-surf}, Item~\ref{prop:horizontal-surf4}, $T_{ij} = T_iT_j^{-1}$. The surface is horizontal if, and only if, each $T_{ij}$ is symmetric, that is, $T_iT_j^{-1} = T_j T_i^{-1}$, which, in turn, is equivalent to $T_iT_j = T_jT_i$.
\end{enumerate}
\end{proof}
Unlike the two-points case, the commutativity condition puts severe restrictions on the set of matrices $\Sigma_0$,...,$\Sigma_k$ generating a geodesic surface, when $k>1$. For instance, if $\Sigma_0=I$, then we have $T_i=\Sigma_i^{1/2}$. Hence, Corollary 9 entails that the matrices $I$,$\Sigma_1$,...,$\Sigma_k$ generate a geodesic surface if, and only if, they pairwise commute.  
 \section{Wasserstein Riemannian exponential\label{sec:riem-expon}}

We aim now at reformulating a Riemannian geodesic in terms of the exponential
map. In other words, the purpose is that of writing the geodesic arc passing through a given
point and having a given velocity at the point itself.

The velocity of the geodesic of Eq.~\eqref{eq:geodesic-new} is 
\begin{equation*}
\dot{\Sigma}(\theta)=(T-I)\Sigma_{0}+\Sigma_{0}(T-I)+2\theta(T-I)\Sigma_{0}(T-I) \ .
\end{equation*}

Using the horizontal lift $\Sigma(\theta) = A(\theta)A^*(\theta)$, the velocity turns out to be
\begin{equation*}
\dot\Sigma(\theta) = \dot A(\theta) A^*(\theta) + A(\theta) {\dot A}^*(\theta) = \dot A(\theta) A^{-1}(\theta) \Sigma(\theta) + \Sigma(\theta) {A}^*(\theta)^{-1} {\dot A}^*(\theta) \ ,
\end{equation*}
where $\dot A(\theta) A^{-1}(\theta) \in \sym n$ by  Eq.~\eqref{eq:geodesicx}. Therefore,
\begin{equation*}
\dot A(\theta) A^{-1}(\theta) = {A}^*(\theta)^{-1} {\dot A}^*(\theta) = \lyapunov {\Sigma(\theta)}{\dot \Sigma(\theta)} \ .
\end{equation*}
In particular, the initial velocity is 
\begin{equation}\label{eq:ML}
\dot{\Sigma}(0)=(T-I)\Sigma(0)+\Sigma(0)(T-I) \ .  
\end{equation}
and $T - I = \lyapunov {\Sigma(0)}{\dot \Sigma(0)}$.

Let us compute the norm of the velocity in the Riemannian metric. The value of $W^2(\dot\Sigma,\dot\Sigma)$ at $\Sigma(\theta)$ is

\begin{multline*}
\traceof{\lyapunov {\Sigma(\theta)}{\dot \Sigma(\theta)}\Sigma(\theta)\lyapunov {\Sigma(\theta)}{\dot \Sigma(\theta)}} = \\ \traceof{\dot A(\theta) A^{-1}(\theta) A(\theta)A^*(\theta) {A}^*(\theta)^{-1} {\dot A}^*(\theta)} = \\ \traceof{\dot A(\theta){\dot A}^*(\theta)} = \traceof{(T-I)\Sigma(0)(T-I)} \ .
\end{multline*}
It is constant, as we expect from the definition by isometric submersion. Also, we can confirm that the length of the geodesic is 

\begin{multline*}
\sqrt{\traceof{(T-I)\Sigma(0)(T-I)}} = \sqrt{\traceof{\Sigma_0 + \Sigma_1 + T\Sigma_0 + \Sigma_0T}} = \\ \sqrt{\traceof{\Sigma_0 + \Sigma_1 + 2(\Sigma_0^{1/2}\Sigma_1\Sigma_0^{1/2})^{1/2}}} \ .
\end{multline*}
The last equality follows from the relation $\Sigma_0^{1/2} T \Sigma_0^{1/2} = (\Sigma_0^{1/2}\Sigma_1\Sigma_0^{1/2})^{1/2}$.

By substituting Eq.~\eqref{eq:ML} into the equation of the geodesic \eqref{eq:geodesic-new}, we get

\begin{multline*}
\Sigma(\theta)  =\Sigma(0)+\theta\left[ (T-I)\Sigma(0)+\Sigma(0)(T-I)\right]
+\theta^{2}(T-I)\Sigma(0)(T-I)  \label{eq:GH} \\
 =\Sigma(0)+\theta\dot{\Sigma}(0)+\theta^{2}\mathcal{L}_{\Sigma(0)}[\dot{\Sigma }
(0)]\Sigma(0)\mathcal{L}_{\Sigma(0)}[\dot{\Sigma}(0)]\ . 
\end{multline*}

We are so led to the following definition, see \cite[p. 101--102]{absil|mahony|sepulchre:2008}) for example.

\begin{definition}
For any $C\in \operatorname{Sym}^{++}\left( n\right) $ and $V\in \operatorname{Sym}\left(
n\right) \simeq T_{C}\operatorname{Sym}^{++}\left( n\right) $, the Wasserstein
Riemannian exponential is 
\begin{equation} \label{eq:EXP}
\operatorname{Exp}_{C}\left( V\right) =C+V+\mathcal{L}_{C}[V]C\mathcal{L}_{C}[V]=(
\mathcal{L}_{C}[V]+I)C(\mathcal{L}_{C}[V]+I)\ ,  
\end{equation}
\end{definition}

Next proposition collects some properties of the Riemannian exponential.

\begin{proposition}\ \label{Prop:FFA}
\begin{enumerate}
\item All geodesics emanating from a point $C\in\ppsym n$ are of the form $\Sigma (\theta)=\operatorname{Exp}_{C}\left(\theta V\right) $,
with $\theta\in J_{V}$, where $J_{V}$ is the open interval about the
origin: 
\begin{equation*}
J_{V}=\setof{ \theta\in \reals}{I+\theta\lyapunov C V\in \ppsym n}  \ .  \label{eq:UNB}
\end{equation*}
\item The map $V\mapsto \operatorname{Exp}_{C}\left( V\right) ,$ restricted to
the open set 
\begin{equation*}
\Theta =\left\{ V\in \operatorname{Sym}\left( n\right) :I+\mathcal{L}_{C}[V]\in 
\operatorname{Sym}^{++}\left( n\right) \right\} ,
\end{equation*}
is a diffeomorphism of $\Theta $ into $\operatorname{Sym}^{++}\left( n\right) $ with inverse 
\begin{equation*}
\operatorname{Log}_{C}\left( B\right)
= (BC)^{1/2} + (CB)^{1/2} - 2C \ ;
\end{equation*}
\item The derivative of the Riemannian exponential is 
\begin{equation*}
d_{X}\left( V\longmapsto \operatorname{Exp}_{C}\left( V\right) \right) =X+\mathcal{L}
_{C}[X]C\mathcal{L}_{C}[V]+\mathcal{L}_{C}[V]C\mathcal{L}_{C}[X]\ .
\end{equation*}
\end{enumerate}
\end{proposition}

\begin{remark}
Notice that $I + \theta \lyapunov C V = \lyapunov C {\frac12 C^{-1}+\theta V}$ hence, $\theta \in J_V$ if $\frac12 C^{-1} + \theta V \in \ppsym n$. 

Clearly, $0\in J_{V}$ and $\operatorname{Exp}_C(0) = C$ and the maximal open interval containing 0 in which $\operatorname{Exp}_C(\theta V) \in \ppsym n$ is precisely $J_V$. Moreover, the interval $J_{V}$ is
unbounded from the right, i.e., it is of the kind $J_{V}=\left( \bar{\theta}
,+\infty \right) $, provided $V\in\psym n$.
Likewise, $J_{V}=\left( -\infty ,\bar{\theta}\right) $, if $- V\in \psym n$.
Similarly, $\Theta$ is an open set containing the origin and so $V\mapsto 
\operatorname{Exp}_{C}\left( V\right) $ is a local diffeomorphism around the origin.

Since the geodesics are not defined for all the values of the parameter $
t\in \reals$, we infer that the Riemannian manifold $\operatorname{Sym}
^{++}\left( n\right) $ is geodesically incomplete.
Of course this is not a surprising fact: $\operatorname{Sym}^{++}\left( n\right)$ is not a complete metric space, and hence Hopf-Rinow theorem implies that it cannot be geodesically complete, see M.P.~do Carmo \cite{docarmo:1992}.
\end{remark}

\begin{proof}
\begin{enumerate}
\item  Let

\begin{equation*}
\Sigma(\theta) = \operatorname{Exp}_{C}\left(\theta V\right) =C+\theta V+\theta^{2}\mathcal{L}
_{C}[V]C\mathcal{L}_{C}[V] \ , \quad \theta\in J_{V} \ .  
\end{equation*}

Clearly, $\Sigma (0)=C$ and $\dot{\Sigma}(0)=V$.
Pick a scalar $\bar{\theta}\in J_{V}$ and consider the two matrices $\Sigma
\left( 0\right) $ and $\Sigma \left( \bar{\theta}\right) $ belonging to the curve $\Sigma .$
Introduce the new parameterization $\tilde{\Sigma}\left( \tau \right) =\Sigma
\left( \tau \bar{\theta}\right) $, so that $\tilde{\Sigma}\left( 0\right) =\Sigma
\left( 0\right) $ and $\tilde{\Sigma}\left( 1\right) =\Sigma \left( \bar{\theta}
\right) $. We have, 

\begin{equation}\label{eq:partial}
\tilde{\Sigma}\left( \tau \right) =C+\tau (\bar{\theta}V) + \tau ^{2}\lyapunov C {\bar{\theta}V} C \lyapunov C {\bar{\theta}V} \ .
\end{equation}
Setting $\tilde T-I=\lyapunov C {\bar{\theta}V}$, we have $\widetilde T \in \ppsym n$ and 
\begin{equation*}
\tilde T C \tilde T = (I+\lyapunov C {\bar \theta}) C (I+\lyapunov C {\bar \theta}) = \tilde \Sigma(1) \ ,
\end{equation*} 
and the Eq.~\eqref{eq:partial} above becomes 

\begin{multline*}
\tilde{\Sigma}\left( \tau \right) =
C+\tau (\tilde T-I)C+\tau C(\tilde T-I)+\tau ^{2}(\tilde T-I)C(\tilde T-I)= \\
=\left[ \left( 1-\tau \right) I+\tau \tilde T\right] C\left[ \left( 1-\tau
\right) I+\tau \tilde T\right] \ ,
\end{multline*}
which is the geodesic connecting $\Sigma(0) = \tilde \Sigma(0) = C$ to $\tilde \Sigma(1) = \Sigma(\bar\theta)$.

\item By Eq.~\eqref{eq:EXP} the solution to Riccati equation 
\begin{equation*}
\operatorname{Exp}_{C}\left( V\right) =(I+\mathcal{L}_{C}[V])C(I+\mathcal{L}
_{C}[V])=B
\end{equation*}
is 
\begin{equation*}
I+\mathcal{L}_{C}[V]=C^{-1/2}(C^{1/2}BC^{1/2})^{1/2}C^{-1/2}\ 
\end{equation*}
provided $I+\mathcal{L}_{C}[V]\in \operatorname{Sym}^{++}\left( n\right) $. This is
true in a sufficiently small neighborhood $\left\Vert V\right\Vert <r$ of
the origin. The inversion of the operator $\mathcal{L}_{C}[\cdot]$ and Eq.~\eqref{eq:sqrtAB} provide the desired formula for $\operatorname{Log}_{C}\left( B\right)$.

\item The derivative follows from a simple bilinear computation.
\end{enumerate}
\end{proof}

The second order properties of the geodesic and the Riemannian exponential will be discussed in Sec.~\ref{sec:hess-newt-meth}.

\section{Natural gradient\label{sec:natural-gradient}}

We have found the form of the Riemannian metric associated to Wasserstein distance. In turn, the inner product equals the second order approximation of $W^2$. This is a general fact, whose interpretation is based on the discussion of the natural gradient of the metric as solution to the problem
\begin{equation*}
  \label{eq:naturalvialagrangian}
  \begin{cases} \max f(X+H) - f(X) \\ \text{subject to} \\
    W^2(X,X+H) = \text{$\varepsilon$ (small and fixed)}
  \end{cases}
\end{equation*}
which allows the identification of the direction of the maximal increase of the function $f$ with the natural gradient, according to the name introduced by Amari \cite{amari:1998natural}, i.e., the Riemannian gradient as defined below.

The Riemannian gradient is the gradient with respect to the inner product of the metric. We denote by $\nabla$ the gradient with respect to the inner product $\scalarat 2 \cdot \cdot$ and by $\Grad$ the gradient with respect to the Riemannian metric. By Prop. \ref{prop:altWmetric}, $W_\Sigma(X,Y) = \scalarat 2 {\lyapunov \Sigma X} Y$, hence for each smooth scalar field $\phi$ we have
\begin{equation*}
  \label{eq:natu}
\Grad \phi (\Sigma) = \mathcal L_{\Sigma}^{-1} [\nabla \phi(\Sigma)] = \nabla \phi(\Sigma) \Sigma + \Sigma \nabla \phi (\Sigma) \ , 
\end{equation*}
where the second equality follows from the definition of $\mathcal L_\Sigma$. Conversely,
\begin{equation*}
\mathcal{L}_{\Sigma }\left[ \operatorname{grad}\phi (\Sigma )\right] =\nabla \phi
(\Sigma ) \ .
\end{equation*}

The gradient flow of a smooth scalar field $\phi $ is the flow
generated by the vector field
\begin{equation*}
\gamma \mapsto (\gamma , - \operatorname{grad}\phi
(\gamma )) \ ,
\end{equation*}  
that is, the flow of the differential equation 
\begin{equation*}
\dot{\gamma}(\theta)=-\operatorname{grad}\phi (\gamma (\theta))=-\left( \nabla \phi (\gamma
(\theta))\gamma (\theta)+\gamma (\theta)\nabla \phi (\gamma (\theta))\right) \ .
\end{equation*}

The gradient flow equation is the model for many optimization problems which are based on various discrete time approximations of the gradient flow. It should be noted that the expression of the natural gradient in the Wasserstein Riemannian metric is simple and does not require any time-consuming operation as it is the case in optimization methods using the Fisher Riemannian metric. We do not discuss this issue here and refer to \cite{amari:1998natural,absil|mahony|sepulchre:2008,malago|pistone:2015FOGA}. 

\subsection{Gradient flow and optimization}\label{sec:optimization}

With reference to the full Gaussian distribution, one can consider smooth functions
defined on ${\reals}^n\times \operatorname{Sym}^{++}\left( n\right) $. The first
component of the gradient does not require a special gradient as the
Riemannian structure is the Euclidean one. The full gradient will thus have
two components: 
\begin{multline} \label{eq:ZAM} 
\Grad  \phi (\mu ,\Sigma ) = 
\left( \nabla _{1}\phi (\mu ,\Sigma ),\operatorname{grad}_{2}\phi (\mu
,\Sigma )\right) = \\ \left( \nabla _{1}\phi (\mu ,\Sigma ),\nabla _{2}\phi (\mu
,\Sigma )\Sigma +\Sigma \nabla _{2}\phi (\mu ,\Sigma )\right) \ .
\end{multline}

An important example is based to the gradient flow of the mean value of an objective function $f \colon \reals^n \to \reals$. Its Euler scheme is used in optimization, see \cite[Ch. 4]{absil|mahony|sepulchre:2008} and \cite{malago|pistone:2014Entropy}. In the second example in Sec.~\ref{sec:entropy} we discuss the gradient flow of the entropy function of a centered Gaussian.

We call relaxation to the full Gaussian model of the objective function $f:{\reals}^{n}\rightarrow {\reals}$ the function
\begin{equation*}
\phi (\mu ,\Sigma )=\mathbb{E}\left[ f(X)\right] ,\quad X\sim \operatorname{N}
_{n}\left( \mu ,\Sigma \right) \ .
\end{equation*}

If we would include the Dirac measures in the Gaussian model, then $
f(x)=\phi(x,0)$ and the function $\phi$ would actually be an extension of
the given function. However, we consider only $\Sigma\in\operatorname{Sym}
^{++}\left( n\right) $ in order to work with a function defined on our
manifold.

There are two ways to calculate the expected value as a function of $\mu $
and $\Sigma$. Each of them leads to a peculiar expression of the natural gradient.

The first one arises from the relation 
\begin{equation*}
\phi (\mu ,\Sigma )=\mathbb{E}\left[ f(\Sigma ^{1/2}Z+\mu )\right] ,\quad
Z\sim \operatorname{N}_{n}\left( 0,I\right) \ .  \label{eq:DET}
\end{equation*}
which will lead to an equation for the gradient involving the derivatives of 
$f$. The second one uses 
\begin{equation*}
\phi (\mu ,\Sigma )=\int f(x)(2\pi )^{-n/2}\operatorname{det}\left( \Sigma \right)
^{-1/2}\operatorname{exp}\left( -\frac{1}{2}(x-\mu )^{\ast }\Sigma ^{-1}(x-\mu
)\right) \ dx\ .  \label{eq:AQ}
\end{equation*}
In this second case the natural gradient will be achieved by an equation not
involving the gradient of the function $f$. Both forms have their own field
of application.

Let us start with Case~\eqref{eq:DET}. Under standard conditions regarding
the derivation under the expectation sign, we have 
\begin{equation*}
\nabla_{1}\phi(\mu,\Sigma)=\mathbb{E}\left[ \nabla f(\Sigma^{1/2}Z+\mu)
\right] =\mathbb{E}\left[ \nabla f(X)\right] \ .
\end{equation*}

By means of Eq.~\eqref{eq:DRO}, it is straightforward to compute  $d_{U}\left( \Sigma \mapsto \phi (\mu ,\Sigma )\right) $.

Note that $\nabla f$ is the column vector and so $\nabla ^{\ast }f$ will be
a row vector. We have
\begin{equation*}
\begin{aligned}
d_{U}\phi (\mu ,\Sigma )& =\mathbb{E}\left[ df(\Sigma ^{1/2}Z+\mu )[\mathcal{
L}_{\Sigma ^{1/2}}\left( U\right) Z]\right] =\mathbb{E}\left[ \nabla ^{\ast
}f(\Sigma ^{1/2}Z+\mu )\mathcal{L}_{\Sigma ^{1/2}}\left( U\right) Z\right] \\
& =\mathbb{E}\left[ \operatorname{Tr}\nabla ^{\ast }f(\Sigma ^{1/2}Z+\mu )\mathcal{L}
_{\Sigma ^{1/2}}\left( U\right) Z\right] .
\end{aligned}
\end{equation*}
Under symmetrization (and setting $X=\Sigma ^{1/2}Z+\mu $): 
\begin{equation*}
\begin{aligned}
d_{U}\phi (\mu ,\Sigma )& =\frac{1}{2}\mathbb{E}\left[ \operatorname{Tr}\mathcal{L}
_{\Sigma ^{1/2}}\left( U\right) \left( Z\nabla ^{\ast }f(X)+\nabla
f(X)Z\right) \right] \\
& =\left\langle U,\mathbb{E}\left( \left( Z\nabla ^{\ast }f(X)+\nabla
f(X)Z\right) \right) \right\rangle _{\Sigma ^{1/2}} \\
& =\frac{1}{2}\mathbb{E}\operatorname{Tr}\mathcal{L}_{\Sigma ^{1/2}}\left( Z\nabla
^{\ast }f(X)+\nabla f(X)Z\right) U \\
& =\left\langle \mathbb{E}\mathcal{L}_{\Sigma ^{1/2}}\left( Z\nabla ^{\ast
}f(X)+\nabla f(X)Z\right) ,U\right\rangle _{2} \ .
\end{aligned}
\end{equation*}
It follows that 
\begin{equation*}
\nabla _{2}\phi (\mu ,\Sigma )=\mathbb{E}\left[\mathcal{L}_{\Sigma ^{1/2}}\left(
Z\nabla ^{\ast }f(X)+\nabla f(X)Z\right)\right] .
\end{equation*}
Calculating the natural gradient:
\begin{multline*}
\operatorname{grad}_{2}\phi (\mu ,\Sigma ) =  \\ \Sigma \mathbb{E}\left[\mathcal{L}_{\Sigma
^{1/2}}\left( Z\nabla ^{\ast }f(X)+\nabla f(X)Z\right)\right] + \mathbb{E}\left[\mathcal{L}_{\Sigma ^{1/2}}\left( Z\nabla ^{\ast }f(X)+\nabla
f(X)Z\right)\right] \Sigma .
\end{multline*}
If we set $\Xi =\mathbb{E}\left[ Z\nabla ^{\ast }f(X)+\nabla f(X)Z\right] $,
the natural gradient admits the representation
\begin{equation*}
\operatorname{grad}_{2}\phi (\mu ,\Sigma )=\Sigma \mathcal{L}_{\Sigma ^{1/2}}\left(
\Xi \right) +\mathcal{L}_{\Sigma ^{1/2}}\left( \Xi \right) \Sigma .
\end{equation*}

We move on to consider the second Case~\eqref{eq:AQ}. Following the standard
computation of the Fisher score and starting from the log-density $p(x;\mu
,\Sigma )$ of $\operatorname{N}_{n}\left( \mu ,\Sigma \right) $, we have 
\begin{equation}
\begin{aligned}
& \left. \log p(x;\mu ,\Sigma )=-\frac{n}{2}\log 2\pi -\frac{1}{2}\log \det
\Sigma -\frac{1}{2}(x-\mu )^{\ast }\Sigma ^{-1}(x-\mu )\right.
\label{eq:logdensity} \\
& \left. =-\frac{n}{2}\log 2\pi -\frac{1}{2}\log \det \Sigma -\frac{1}{2} - \operatorname{Tr}\left( \Sigma ^{-1}(x-\mu )(x-\mu )^{\ast }\right) \ .\right. 
\end{aligned}
\end{equation}

Denoting the partial derivative $d_{u}\left( \mu \longmapsto \log p(x;\mu
,\Sigma )\right) $ as $d_{u}\log p(x;\mu ,\Sigma )$, and the other derivative $d_{U}\left( \Sigma
\longmapsto \log p(x;\mu ,\Sigma )\right) $ as $d_{U}\log p(x;\mu ,\Sigma )$%
, we get:
\begin{equation*}
\begin{aligned}
d_{u}\log p(x;\mu ,\Sigma )& =(x-\mu )^{\ast }\Sigma ^{-1}u=\left\langle
\Sigma ^{-1}(x-\mu ),u\right\rangle \\
d_{U}\log p(x;\mu ,\Sigma )& =-\frac{1}{2}\operatorname{Tr}\left( \Sigma
^{-1}U\right) +\frac{1}{2}\operatorname{Tr}\left( \Sigma ^{-1}U\Sigma ^{-1}(x-\mu
)(x-\mu )^{\ast }\right) \\
& =\frac{1}{2}\left\langle \Sigma ^{-1}(x-\mu )(x-\mu )^{\ast }\Sigma
^{-1}-\Sigma ^{-1},U\right\rangle \\
& =\left\langle \Sigma ^{-1}\left( (x-\mu )(x-\mu )^{\ast }-\Sigma \right)
\Sigma ^{-1},U\right\rangle _{2}
\end{aligned}
\end{equation*}
So that 
\begin{equation*}
\begin{aligned}
d_{u}\phi (\mu ,\Sigma )& =\int f(x)\ d_{u}\log p(x;\mu ,\Sigma )\ p(x;\mu
;\Sigma )\ dx \\
& =\left\langle \Sigma ^{-1}\int f(x)(x-\mu )p(x;\mu ;\Sigma )\
dx,u\right\rangle
\end{aligned}
\end{equation*}
and 
\begin{equation*}
\begin{aligned}
d_{U}\phi (\mu ,\Sigma )& =\int f(x)\ d_{U}\log p(x;\mu ,\Sigma )\ p(x;\mu
,\Sigma )\ dx \\
& =\left\langle \Sigma ^{-1}\int f(x)\left( (x-\mu )(x-\mu )^{\ast }-\Sigma
\right) p(x;\mu ,\Sigma )\ dx\ \Sigma ^{-1},U\right\rangle _{2}.
\end{aligned}
\end{equation*}
At last, thanks to Eq.~\eqref{eq:ZAM}, the natural gradient of $\phi (\mu
,\Sigma )$ will be 
\begin{equation*}
\begin{aligned}
\nabla _{1}\phi (\mu ,\Sigma )& =\Sigma ^{-1}\int f(x)(x-\mu )p(x;\mu
;\Sigma )\ dx \\
\operatorname{grad}_{2}\phi (\mu ,\Sigma )& =\int f(x)\left( (x-\mu )(x-\mu )^{\ast
}-\Sigma \right) p(x;\mu ,\Sigma )\ dx\ \Sigma ^{-1} \\
& +\Sigma ^{-1}\int f(x)\left( (x-\mu )(x-\mu )^{\ast }-\Sigma \right)
p(x;\mu ,\Sigma )\ dx.
\end{aligned}
\end{equation*}

\subsection{Entropy gradient flow\label{sec:entropy}}

The flow of entropy can be easily calculated by Eq.~\eqref{eq:logdensity}. We
have 
\begin{equation*}
\begin{aligned}
\mathcal{E}(\mu,\Sigma) & =-\int\log p(x;\mu,\Sigma)p(x;\mu,\Sigma)\ dx \\
& =\frac{n}{2}\log2\pi+\frac{1}{2}\log\det\Sigma-\frac{1}{2}\operatorname{Tr}\left(
\Sigma^{-1}\Sigma\right) \\
& =\frac{n}{2}(\log2\pi-1)+\frac{1}{2}\log\det\Sigma\ .
\end{aligned}
\end{equation*}

The entropy does not depend on $\mu$ so that $\nabla_{1}\mathcal{E}(\mu
,\Sigma)=0$. Moreover (see \cite[\S 8.3]{magnus|neudecker:1999}) we know
that $\nabla\mathcal{E}(\Sigma)=\Sigma^{-1}$, so that 
\begin{equation*}
\operatorname{grad}\mathcal{E}(\Sigma)=(\Sigma^{-1}\Sigma+\Sigma\Sigma ^{-1})=2I.
\end{equation*}
The entropic flow will be solution to the equations 
\begin{equation*}
\dot{\mu}(t)=0,\quad\dot{\Sigma}(t)+2I=0\ ,
\end{equation*}
that is 
\begin{equation*}
\mu(t)=\mu(0),\quad\Sigma(t)=\Sigma(0)-2tI\ .
\end{equation*}

The integral curve is defined for all $t$ such that $2t<\lambda_{*}$, $%
\lambda_{*}$ being the minimum of the spectrum of $\Sigma(0)$.

\section{Second order geometry}\label{sec:second-order}

Recall that $\ppsym n$ as an open set of the Hilbert space $\sym n$, endowed with the inner product $\scalarat 2 X Y = \frac12 \traceof{XY}$. Prop.~\ref{prop:altWmetric} shows that the Wasserstein Riemannian metric $W$ can be expressed in terms of the inner product of $\sym n$ by
\begin{equation*}
W_\Sigma(X,Y) = \scalarat \Sigma X Y = \scalarat 2 {\lyapunov \Sigma X} Y \ ,  
\end{equation*}
for each $(\Sigma,X)$ and $(\Sigma,Y)$ in the trivial tangent bundle $T \ppsym n \simeq \ppsym n \times \sym n$. In the equation above, $\mathcal L \colon \ppsym n \mapsto L(\sym n,\sym n)$ is the field of linear operators defining the Wasserstein metric with respect to the standard inner product.  

In the trivial chart, a smooth vector field $X$ is a smooth mapping $X \colon \ppsym n \to \sym n$. The action of the vector field $X$ on the scalar field $f$ that is, $Xf$, is expressed in the trivial chart by $d_Xf$, i.e., the scalar field whose value at point $\Sigma$ is the derivative of $f$ in the direction $X(\Sigma)$. Similarly, $d_YX$ denotes the vector field whose value at point $\Sigma$ is the derivative at $\Sigma$ of $X$ in the direction $Y(\Sigma)$. The Lie bracket $[X,Y]$ of two smooth vector fields $X,Y$ is given by $d_XY - d_YX$.

\subsection{The moving frame}
While we prefer to express our computation by matrix algebra, in some cases it may be useful to employ a vector basis. We discuss below a field of vector bases  of particular interest.

The set of symmetric matrices
\begin{equation*}\label{eq:Epq}
  E^{p,q} = e_p e_q^* + e_qe_p^*, \quad p,q = 1,\dots,n \ ,
\end{equation*}
$e_p$ being the $p$-th element of the standard basis of $\reals^n$, spans the vector space $\sym n$. Notice that $\traceof{E^{p,q}} = 2 \delta_{p,q}$, where $\delta$ is the Kronecker symbol. To avoid repeated elements, a unique enumeration is obtained by taking  indexes in the set $A$ of the parts of $\set{1,\dots,n}$ having 1 or 2 elements. 

The generating set of Eq.~\eqref{eq:Epq} is related to the symmetric product of matrices by the equation
\begin{equation*}
 E^{p,q} E^{r,s}  + E^{r,s} E^{p,q} = \delta_{q,r} E^{p,s} + \delta_{q,s}E^{p,r} + \delta_{p,r}E^{q,s} + \delta_{p,s} E^{q,r} \ ,
\end{equation*}
where $\delta$ is the Kronecker symbol.

In particular, if we take the trace of the equation above, we get
\begin{equation*}
\scalarat 2 {E^{p,q}} {E^{r,s}} = \delta_{p,r}\delta_{q,s}+\delta_{p,s}\delta_{q,r} \ ,
\end{equation*}
which in turn implies
\begin{equation*}
 \scalarat 2 {E^{p,q}} {E^{r,s}} = \begin{cases}
   0 & \text{if $\set{p,q}\neq\set{r,s}$}, \\
   1 & \text{if $\set{p,q}=\set{r,s}$ and $p \ne q$}, \\
   2 & \text{if $\set{p,q}=\set{r,s}$ and $p = q$} \\
  \end{cases}
\end{equation*}

In the sequel, we denote by $(E^\alpha)_{\alpha \in A}$ the vector basis above, properly normalized to obtain an orthonormal basis. We do not write down the normalizing constants in order to simplify the notation.

For each $\Sigma \in \ppsym n$ the sequence
\begin{equation}\label{eq:movingframe}
\MF \alpha(\Sigma) = E^\alpha \Sigma + \Sigma E^\alpha , \quad \alpha \in A \ ,   
\end{equation}
is a vector basis of $\sym n \simeq T_\Sigma \ppsym n$, because it is the image of a vector basis under a linear mapping which is onto. We will call such a sequence of vector fields the (principal) moving frame. 

Notice the following properties:
\begin{equation*}
\MF \alpha = d_{E^\alpha}\Sigma^2 \ ; \quad
\lyapunov \Sigma {\MF\alpha(\Sigma)} = E^\alpha \ ; \quad
\MF\alpha(I) = 2E^\alpha \ .
\end{equation*}

At a generic point $\Sigma$, we can express each $\MF\alpha$ in the $(E^\beta)_\beta$'s orthonormal basis as
\begin{equation}  \label{eq:Ginv}
\MF \alpha (\Sigma) = \sum_\beta g_{\alpha,\beta}(\Sigma) E^\beta \ , \quad g_{\alpha,\beta}(\Sigma) = \traceof{E^\alpha \Sigma E^\beta} \ .  
\end{equation}
Since
\begin{equation*}
W_\Sigma(\MF \alpha,\MF \beta) = \traceof{\lyapunov \Sigma {\MF \alpha (\Sigma)} \Sigma \lyapunov \Sigma {\MF \beta (\Sigma)}} = \traceof{E^\alpha \Sigma E^\beta} \ ,
\end{equation*}
the matrix $[g_{\alpha,\beta}]_{\alpha,\beta}$ is the expression of the Riemannian metric in such a  moving frame. Namely, if $X,Y$ are vector fields expressed in the moving frame as $X = \sum_\alpha x_\alpha \MF \alpha$ and $Y = \sum_\beta y_\beta \MF \beta$, then

\begin{multline*}
  W_\Sigma(X,Y) = \traceof{\lyapunov \Sigma {\sum_\alpha x_\alpha(\Sigma)\MF \alpha} \Sigma (\Sigma) \ \lyapunov \Sigma {\sum_\beta y_\beta(\Sigma) \MF \beta} (\Sigma)} = \\ \traceof{\left(\sum_\alpha x_\alpha(\Sigma)E^\alpha\right) \Sigma \left(\sum_\beta y_\beta(\Sigma) E^\beta\right)} = \sum_{\alpha,\beta} x_\alpha(\Sigma)y_\beta(\Sigma)g_{\alpha,\beta}(\Sigma) \ .
\end{multline*}
This expression of the inner product is to be compared to that used in \cite{takatsu:2011osaka}. 

In this way, any vector field $X$ has two representations: one with respect to the moving frame $(\MF \alpha)_\alpha$ and another one with respect to the basis $(E^\alpha)_\alpha$. These two representations are related to each other as follows. We have
\begin{equation*}
X = \sum_\alpha x_\alpha \MF \alpha = \sum_\alpha x_\alpha \sum_\beta g_{\alpha,\beta} E^\beta = \sum_\beta \left(\sum_\alpha x_\alpha g_{\alpha,\beta}\right) E^\beta \ ,  
\end{equation*}
so that
\begin{equation*}
  \scalarat 2 X {E^\gamma} = \frac12
\traceof{XE^\gamma} = \sum_\beta \left(\sum_\alpha x_\alpha g_{\alpha,\beta}\right) \traceof{E^\beta E^\gamma} = \sum_\alpha x_\alpha g_{\alpha,\gamma} \ ,    
\end{equation*}
hence, by applying the inverse matrix $[g^{\alpha,\beta}(\Sigma)]=[g_{\alpha,\beta}(\Sigma)]^{-1}$, we have
\begin{equation}\label{eq:movingformfixed}
x_\alpha = \sum_\gamma g^{\alpha,\gamma} \scalarat 2 {X}{E^\gamma} \ .    
\end{equation}

For example, $\lyapunov \Sigma V = \sum_\alpha \ell_{\Sigma}^\alpha(V) \MF \alpha (\Sigma)$, with
\begin{equation*}
\ell_{\Sigma}^\alpha(V) = \sum_\gamma g^{\alpha,\gamma}(\Sigma)\scalarat 2 {\lyapunov \Sigma V} {E^\gamma} =  W_\Sigma(V,\sum_\gamma g^{\alpha,\gamma}E^\gamma) \ .   
\end{equation*}

\subsection{Covariant derivative in the moving frame}
If $X$ and $Y$ are vector fields, denote by $D_YX$ the action of a covariant derivative, namely, a bilinear operator satisfying, for each scalar field $f$, the following two conditions:

\begin{itemize}
\item[\it(CD1)]
$D_{fY}X = fD_Y X$ \ ,
\item[\it(CD2)]
$D_Y (fX) = (d_Yf) X + f D_Y X$ \ .
\end{itemize}
see e.g \cite[Sect. 3]{docarmo:1992} or \cite[Ch. 8.4]{lang:1995}. 

A convenient way to express a covariant derivative in the moving frame \eqref{eq:movingframe} is to define Christoffel symbols in the moving frame as
\begin{equation*}
  \sum_\gamma \Gamma_{\alpha,\beta}^\gamma \MF\gamma = D_{\MF \alpha} \MF \beta = E^\beta E^\alpha + E^\alpha E^\beta \ .
\end{equation*}
Each $\Gamma_{\alpha,\beta}^\gamma$ is to be computed by means of Eq.~\eqref{eq:movingformfixed}.

If $X=\sum_\alpha x_\alpha \MF \alpha$ and $Y = \sum_\beta y_\beta \MF \beta$, by using \emph{(CD1)}, \emph{(CD2)}, and Eq. \eqref{eq:Ginv}, we obtain  

\begin{multline*}
  D_XY = \sum_{\alpha,\beta} x_\alpha D_{\MF\alpha} (y_\beta \MF\beta) = \sum_{\alpha,\beta} x_\alpha \left(\left(d_{\MF\alpha}y_\beta\right) \MF\beta + y_\beta \left(D_{\MF\alpha}\MF\beta\right)\right) = \\
\sum_{\alpha,\gamma} x_\alpha d_{\MF\alpha}y_\gamma \MF\gamma + \sum_{\alpha,\beta,\gamma} y_\beta \Gamma_{\alpha,\beta}^\gamma \MF\gamma  = 
 \sum_\gamma \sum_{\alpha,\beta} x_\alpha \left(d_{\MF\alpha}y_\gamma  + y_\beta \Gamma_{\alpha,\beta}^\gamma \right)\MF\gamma \ .
\end{multline*}
 
The inner product of $D_XY$ and $Z = \sum_\delta z_\delta \MF\delta$ is
\begin{equation*}
  \scalarat \Sigma {D_XY} Z = \sum_{\alpha,\beta,\gamma,\delta}  x_\alpha \left(d_{\MF\alpha}y_\gamma  + y_\beta \Gamma_{\alpha,\beta}^\gamma \right) g_{\delta,\gamma} z_\delta\ .
  \end{equation*}

\subsection{Levi-Civita derivative\label{sec:levi-civita-covar}}

The Levi-Civita (covariant) derivative of a vector field, is the unique covariant derivative $D$ that, for all vector fields $X,Y,Z$, is

\begin{tabular}{ll}
\emph{(LC1)} & compatible with the metric, $d_{X}W(Y,Z)=W(D_X Y,Z) + W(Y,D_{X}Z)$, \\
\emph{(LC2)} & torsion-free, $D_{Y}X-D_{X}Y = [X,Y] = d_Y X - d_X Y$. \end{tabular}

In order to keep a compact notation, it will be convenient to make use of the
symmetrized of a matrix $A \in \matrices n$, defined by $\symmetricof A = \frac{1}{2}\left( A+A^*\right)$. If either $A$ or $B$ is symmetric, then $\traceof{\symmetricof{A}B} = \traceof{AB}$. We denote by $X,Y,Z$ smooth vector fields on $\ppsym n$. We shall use repeatedly the expression for the derivative of the vector field $\Sigma \mapsto \lyapunov \Sigma X$. In view of Eq.~\eqref{eq:DEDE} and under our notation for the symmetrization, it holds
\begin{equation*}
  \label{eq:DL1}
  d_Y \lyapunov \Sigma X = -2 \lyapunov \Sigma {\symmetricof{\lyapunov \Sigma X Y}} \ .
\end{equation*}

\begin{proposition}\label{th:levi}
  The Levi-Civita derivative $D_{X}Y$ is implicitly defined by

\begin{multline} \label{eq:HP} 
\scalarat \Sigma {D_{X}Y} Z = \scalarat \Sigma {d_{X}Y} Z +  \scalarat \Sigma X {\symmetricof{\lyapunov \Sigma Y Z}} \\ - \scalarat \Sigma X {\symmetricof{\lyapunov \Sigma Z Y}} - \scalarat \Sigma Y {\symmetricof{\lyapunov \Sigma Z X}}
= \\ \scalarat \Sigma {d_{X}Y} Z + \frac{1}{2}\traceof{\lyapunov \Sigma X Z \lyapunov \Sigma Y} - \\ \frac{1}{2}\traceof{\lyapunov \Sigma X Y \lyapunov \Sigma Z} - \frac{1}{2}\traceof{\lyapunov \Sigma Y X \lyapunov \Sigma Z} \ ,
\end{multline}
while the Levi-Civita derivative itself is given by
\begin{equation*}  \label{eq:OP} 
D_{X}Y = d_{X}Y - \symmetricof{\lyapunov \Sigma X Y + \lyapunov \Sigma Y X} + \symmetricof{\Sigma \lyapunov \Sigma X \lyapunov \Sigma Y + \Sigma \lyapunov \Sigma Y \lyapunov \Sigma X} \ . 
\end{equation*}
\end{proposition}

\begin{proof}
In our case, Eq. \emph{MD3} of \cite[p. 205]{lang:1995} becomes

\begin{multline}\label{eq:GHA} 
2 \scalarat 2 {D_X Y}{\lyapunov \Sigma Z} = \\ 2\scalarat 2 {d_XY}{\lyapunov \Sigma Z} + \scalarat 2 {Y}{d_X \lyapunov \Sigma Z} + \scalarat 2 {X}{d_Y \lyapunov \Sigma Z} - \scalarat 2 {X}{d_Z \lyapunov \Sigma Y} \ .
\end{multline}

By Eq.~\eqref{eq:DEDE} we have 
\begin{equation*}
  \scalarat 2 {Y}{d_X \lyapunov \Sigma Z} = - 2 \scalarat 2 {Y} {\lyapunov \Sigma {\symmetricof{\lyapunov \Sigma Z X}}} = - 2 \scalarat \Sigma {Y} {\symmetricof {\lyapunov \Sigma Z X}} \ ,
\end{equation*}
and, analogously, 
\begin{equation*}
  \scalarat 2 {X}{d_Y \lyapunov \Sigma Z} = - 2 \scalarat \Sigma {X} {\symmetricof {\lyapunov \Sigma Z Y}}, \quad 
  \scalarat 2 {X}{d_Z \lyapunov \Sigma Y} = - 2 \scalarat \Sigma {X} {\symmetricof {\lyapunov \Sigma Y Z}} \ .
\end{equation*}
This way, Eq.~\eqref{eq:GHA} becomes  the first part of Eq.~\eqref{eq:HP}.

The second part of Eq.~\eqref{eq:HP} is then easily obtained. For instance,
\begin{equation*}
\scalarat \Sigma X {\symmetricof{\lyapunov \Sigma Z}} = \frac12 \traceof{\lyapunov \Sigma X \symmetricof{Z \lyapunov \Sigma Y}} =\frac12 \traceof{\lyapunov \Sigma X Z  \lyapunov \Sigma Y} \ .  
\end{equation*}

Regarding the explicit formula of the Levi-Civita derivative \eqref{eq:OP},
observe that

\begin{multline*}
\frac12 \traceof{\lyapunov \Sigma X Z \lyapunov \Sigma Y} = \frac12 \traceof{\lyapunov \Sigma Y \lyapunov \Sigma X Z } = \frac12 \traceof{\symmetricof{\lyapunov \Sigma X \lyapunov \Sigma Y} Z} = \\ \frac12 \traceof{\lyapunov \Sigma {\symmetricof{\lyapunov \Sigma X \lyapunov \Sigma Y}\Sigma + \Sigma\symmetricof{\lyapunov \Sigma X \lyapunov \Sigma Y}}Z} = \\
\scalarat \Sigma {\symmetricof{\lyapunov \Sigma X \lyapunov \Sigma Y}\Sigma + \Sigma\symmetricof{\lyapunov \Sigma X \lyapunov \Sigma Y}}Z = \\
\scalarat \Sigma {\symmetricof{\Sigma \lyapunov \Sigma X \lyapunov \Sigma Y} + \symmetricof{\Sigma \lyapunov \Sigma Y \lyapunov \Sigma X}}{Z} = \\ \scalarat \Sigma {\symmetricof{\Sigma \lyapunov \Sigma X \lyapunov \Sigma Y + \Sigma \lyapunov \Sigma Y \lyapunov \Sigma X}}{Z} \ .
\end{multline*}
Moreover,

\begin{multline*}
  \frac{1}{2}\traceof{\lyapunov \Sigma X Y \lyapunov \Sigma Z} + \frac{1}{2}\traceof{\lyapunov \Sigma Y X \lyapunov \Sigma Z} = \\
  \frac12 \traceof{\symmetricof{\lyapunov \Sigma X Y + \lyapunov \Sigma Y X} \lyapunov \Sigma Z} =
  \scalarat \Sigma {\symmetricof{\lyapunov \Sigma X Y + \lyapunov \Sigma Y X}} Z
 \ .\end{multline*}
Therefore, Eq.~\eqref{eq:HP} can be written as 

\begin{multline*}
\scalarat \Sigma {D_XY} Z = \\ \scalarat \Sigma {d_{X}Y - \symmetricof{\lyapunov \Sigma X Y + \lyapunov \Sigma Y X} + \symmetricof{\Sigma \lyapunov \Sigma X \lyapunov \Sigma Y + \Sigma \lyapunov \Sigma Y \lyapunov \Sigma X}} Z \ ,
\end{multline*}
and the desired result obtains.
\end{proof}

We have computed the Levi-Civita covariant derivative using its explicit expression in term of derivatives of the metric. However is easy to check the result directly using the properties of the Lyapunov operator.

\subsection{Levi-Civita derivative in a moving frame}

Let us express the Levi-Civita derivative in the moving frame \eqref{eq:movingframe}. Note that $X(\Sigma) = \MF\alpha(\Sigma) = E^\alpha \Sigma + \Sigma E^\alpha$ and $Y(\Sigma) = \MF\beta(\Sigma) = E^\beta\Sigma + \Sigma E^\beta$ are vector fields. 

\begin{proposition}
For the Levi-Civita covariant derivative $D$, it holds
\begin{equation*}
D_{\MF\alpha}\MF\beta = E^\beta E^\alpha \Sigma + \Sigma E^\alpha E^\beta \ .
\end{equation*}
\end{proposition}

\begin{proof}
Eq.~\eqref{eq:OP} yields 

\begin{multline}\label{eq:OP-MF}
D_{\MF\alpha}\MF\beta =  d_{\MF\alpha}{\MF\beta} - \symmetricof{\lyapunov \Sigma {\MF\alpha} {\MF\beta} + \lyapunov \Sigma {\MF\beta} {\MF\alpha}} + \\ \symmetricof{\Sigma \lyapunov \Sigma {\MF\alpha} \lyapunov \Sigma {\MF\beta} + \Sigma \lyapunov \Sigma {\MF\beta} \lyapunov \Sigma {\MF\alpha}} \ . 
\end{multline}
We are going to compute one by one the three terms in this equation.

The first term of Eq.~\eqref{eq:OP-MF} is

\begin{multline*}
  d_{\MF\alpha} \MF\beta = d_{(E^\alpha \Sigma + \Sigma E^\alpha)}(E^\beta \Sigma + \Sigma E^\beta) = \\
  E^\beta (E^\alpha \Sigma + \Sigma E^\alpha) + (E^\alpha \Sigma + \Sigma E^\alpha) E^\beta = \\ E^\beta E^\alpha \Sigma + E^\beta \Sigma E^\alpha + E^\alpha \Sigma E^\beta + \Sigma E^\alpha E^\beta  \ .
\end{multline*}

The second one is

\begin{multline*}
 - \symmetricof{\lyapunov \Sigma {\MF\alpha} {\MF\beta} + \lyapunov \Sigma {\MF\beta} {\MF\alpha}} = \\ - \symmetricof{E^\alpha(E^\beta \Sigma + \Sigma E^\beta) + E^\beta(E^\alpha \Sigma + \Sigma E^\alpha)} = \\
  - \symmetricof{E^\alpha E^\beta \Sigma + E^\alpha \Sigma E^\beta + E^\beta E^\alpha \Sigma + E^\beta \Sigma E^\alpha} = \\
 - \frac12 \left(E^\alpha E^\beta \Sigma + E^\beta E^\alpha \Sigma + \Sigma E^\beta E^\alpha + \Sigma E^\alpha E^\beta \right) - \left(E^\alpha \Sigma E^\beta  + E^\beta \Sigma E^\alpha\right) \ .
\end{multline*}

Their sum is
\begin{equation*}
  \frac12\left(E^\beta E^\alpha \Sigma + \Sigma E^\alpha E^\beta\right) -
  \frac12\left(E^\alpha E^\beta \Sigma + \Sigma E^\beta E^\alpha\right) \ .
\end{equation*}

The third term is

\begin{multline*}
  \symmetricof{\Sigma \lyapunov \Sigma {\MF\alpha} \lyapunov \Sigma {\MF\beta} + \Sigma \lyapunov \Sigma {\MF\beta} \lyapunov \Sigma {\MF\alpha}} = 
  \symmetricof{\Sigma E^\alpha E^\beta + \Sigma E^\beta E^\alpha} = \\
  \frac12 \left(\Sigma E^\alpha E^\beta + \Sigma E^\beta E^\alpha + E^\beta E^\alpha \Sigma + E^\alpha E^\beta \Sigma\right) \ .
\end{multline*}
\end{proof}

The computation of the Christoffel symbols $\sum_\gamma \Gamma_{\alpha,\beta}^\sigma \MF\gamma = D_{\MF\alpha}\MF\beta $ would require the solution of the equations
\begin{equation*}
  E^\beta E^\alpha \Sigma + \Sigma E^\alpha E^\beta = \sum_\gamma \Gamma_{\alpha,\beta}^\gamma(\Sigma) \left( E^\gamma \Sigma + \Sigma E^\gamma \right) \ .
\end{equation*}
We do not discuss that here. 

Instead, let us take now $X = x_\alpha \MF\alpha$ and $Y=y_\beta \MF\beta$.  Properties \emph{(CD1)} and \emph{(CD2)} lead to

\begin{multline*}
D_{(x_\alpha \MF\alpha)} (y_\beta \MF\beta) = x_\alpha D_{E^\alpha} (y_\beta E^\beta) = x_\alpha \left(d_{E^\alpha}y_\beta E^\beta + y_\beta D_{E^\alpha} E^\beta\right) = \\ 
x_\alpha d_{E^\alpha}y_\beta E^\beta + x_\alpha y_\beta \left(E^\beta E^\alpha \Sigma + \Sigma E^\alpha E^\beta\right) \ .
\end{multline*}

Finally, for general $X$ and $Y$,
\begin{equation*}
D_XY = \sum_{\alpha,\beta} x_\alpha d_{E^\alpha}y_\beta E^\beta + \sum_{\alpha,\beta} x_\alpha y_\beta \left(E^\beta E^\alpha \Sigma + \Sigma E^\alpha E^\beta\right) \ 
\end{equation*}
which is the desired result.

\subsection{Parallel transport\label{sec:parallel-transport}}

The expression of the Levi-Civita derivative in Eq.~\eqref{eq:HP}  can be re-written as
\begin{equation*}
  \label{eq:spray}
\scalarat \Sigma {D_{X}Y} Z =  \scalarat \Sigma {d_{X}Y} Z + \scalarat \Sigma {\Gamma(\Sigma;X,Y)} Z \ ,  
\end{equation*}
where $\Gamma(\Sigma;\cdot,\cdot)$ is the symmetric tensor field
defined by

\begin{multline*} \label{eq:GammavsZ} 
  \scalarat \Sigma {\Gamma(\Sigma;X,Y)} Z = \\ \frac{1}{2}\traceof{\lyapunov \Sigma X Z \lyapunov \Sigma Y} - \frac{1}{2}\traceof{\lyapunov \Sigma X Y \lyapunov \Sigma Z} - \frac{1}{2}\traceof{\lyapunov \Sigma Y X \lyapunov \Sigma Z} = \\
  \frac{1}{2}\traceof{\lyapunov \Sigma Y \lyapunov \Sigma X Z } - \frac{1}{2}\traceof{\left(\lyapunov \Sigma X Y + \lyapunov \Sigma Y X\right) \lyapunov \Sigma Z} = \\
 \frac{1}{2}\traceof{\lyapunov \Sigma Y \lyapunov \Sigma X \left(\lyapunov \Sigma Z \Sigma + \Sigma \lyapunov \Sigma Z\right) } - \frac{1}{2}\traceof{\left(\lyapunov \Sigma X Y + \lyapunov \Sigma Y X\right) \lyapunov \Sigma Z} = \\
  \frac12 \traceof{\left(\Sigma \lyapunov \Sigma Y \lyapunov \Sigma X + \lyapunov \Sigma Y \lyapunov \Sigma X \Sigma - \lyapunov \Sigma X Y - \lyapunov \Sigma Y X \right)\lyapunov \Sigma Z} = \\
\scalarat \Sigma {\symmetricof{\Sigma \lyapunov \Sigma Y \lyapunov \Sigma X + \lyapunov \Sigma Y \lyapunov \Sigma X \Sigma - \lyapunov \Sigma X Y - \lyapunov \Sigma Y X}} {Z} \ .
\end{multline*}

We have
\begin{equation*}
  \label{eq:Gamma}
  \Gamma(\Sigma;X,Y) = \symmetricof{\Sigma \lyapunov \Sigma Y \lyapunov \Sigma X + \lyapunov \Sigma Y \lyapunov \Sigma X \Sigma - \lyapunov \Sigma X Y - \lyapunov \Sigma Y X} \ ,
\end{equation*}
and, on the diagonal,
\begin{equation*}
  \label{eq:Gammadiag}
\Gamma(\Sigma;X,X) =  \Sigma \lyapunov \Sigma  X \lyapunov \Sigma   X + \lyapunov \Sigma X \lyapunov \Sigma  X \Sigma - \lyapunov \Sigma X X - X \lyapunov \Sigma X \ .
\end{equation*}

$\Gamma(\Sigma;X,Y)$ is the expression in the trivial chart of the Christoffel symbol of the Levi-Civita derivative as in \cite{klingenberg:1995}. In  \cite{lang:1995}, $-\Gamma$ is called the spray of the Levi-Civita derivative. 

Given the Christoffel symbol, the linear differential equation of the parallel transport along a curve $t \mapsto \Sigma(t)$ is
\begin{equation*}\label{eq:transportDE}
  \begin{cases}
    \dot U_V(t) + \Gamma(\Sigma(t);\dot \Sigma(t), U_V(t)) = 0 \ , \\
    U_V(0) = V \ ,
  \end{cases}
\end{equation*}
see \cite[VIII, \S 3 and \S 4]{lang:1995}. Recall that the parallel transport for the Levi-Civita derivative is isometric.

We do not discuss here the representation in the moving frame of Eq.~\eqref{eq:transportDE}. We limit ourselves to mention that the action of the Christoffel symbol on vector fields expressed in the moving frame can be computed from

\begin{multline*}
  \label{eq:GammaMF}
  \Gamma(\Sigma;\MF \alpha,\MF \beta) = \\ \symmetricof{\Sigma \lyapunov \Sigma {\MF\beta} \lyapunov \Sigma {\MF\alpha} + \lyapunov \Sigma {\MF\beta} \lyapunov \Sigma {\MF\alpha} \Sigma - \lyapunov \Sigma {\MF\alpha} \MF\beta - \lyapunov \Sigma {\MF\beta} \MF\alpha} = \\
  \symmetricof{\Sigma E^\beta E^\alpha + E^\beta E^\alpha \Sigma - E^\alpha ( E^\beta\Sigma+\Sigma E^\beta) - E^\beta (E^\alpha\Sigma+\Sigma E^\alpha)} = \\
  \symmetricof{\Sigma E^\beta E^\alpha + E^\beta E^\alpha \Sigma - E^\alpha E^\beta\Sigma  - E^\alpha\Sigma E^\beta - E^\beta E^\alpha\Sigma -  E^\beta \Sigma E^\alpha} = \\ - (E^\alpha \Sigma E^\beta + E^\beta \Sigma E^\alpha) \ .
\end{multline*}

\subsection{Riemannian Hessian\label{sec:hess-newt-meth}}

According to \cite[Def. 5.5.1]{absil|mahony|sepulchre:2008} and \cite[p. 141]{docarmo:1992}, the
Riemannian Hessian of a smooth scalar field $\phi \colon \ppsym n \to \reals$,  is the Levi-Civita covariant derivative of the natural gradient $\Grad \phi$. Namely, for each vector field $X$, it is the vector field $\Hessian_X \phi$ whose value at $\Sigma $ is
\begin{equation*}
\operatorname{Hess}_{X}\phi (\Sigma )=D_{X}(\operatorname{grad}\phi )(\Sigma )=D_{X}(\nabla
\phi (\Sigma )\Sigma +\Sigma \nabla \phi (\Sigma )) \ .
\end{equation*}

The associated symmetric bilinear form is (see \cite[Prop.
5.5.3]{absil|mahony|sepulchre:2008})
\begin{equation*}
\operatorname{Hess}\phi (\Sigma )\left( X,Y\right) =\left\langle D_{X}(\operatorname{grad}
\phi )(\Sigma ),Y\right\rangle _{\Sigma } \ .
\end{equation*}

To our purpose it will be enough to compute the diagonal of the symmetric form. Therefore, letting $X=Z=V$ in the second part of Eq. \eqref{eq:HP}, we obtain

\begin{multline*}
\operatorname{Hess}\phi (\Sigma )\left( V,V\right) = \left\langle
d_{V}Y,V\right\rangle _{\Sigma } + \\ \frac{1}{2}\operatorname{Tr}\left[ \mathcal{L}
_{\Sigma }\left[V\right] V\mathcal{L}_{\Sigma }\left[Y\right] \right] -
\frac{1}{2}\operatorname{Tr}\left[ \mathcal{L}_{\Sigma }\left[ V\right] Y\mathcal{L}
_{\Sigma }\left[V\right] \right] 
-\frac{1}{2}\operatorname{Tr}\left[ \mathcal{L}_{\Sigma }\left[ V\right] V\mathcal{
L}_{\Sigma }\left[Y\right] \right] = \\
\left\langle d_{V}Y,V\right\rangle _{\Sigma }-\frac{1}{2}\operatorname{Tr}\left[ 
\mathcal{L}_{\Sigma }\left[V\right] Y\mathcal{L}_{\Sigma }\left[V\right] 
\right] \ ,
\end{multline*}
where $Y=\operatorname{grad}\phi \left( \Sigma \right) $. After plugging $Y = 
\operatorname{grad}\phi \left( \Sigma \right) =\Sigma \nabla \phi \left( \Sigma
\right) +\nabla \phi \left( \Sigma \right) \Sigma $ into it, we get easily

\begin{multline*}
\operatorname{Hess}\phi (\Sigma )\left( V,V\right) = \left\langle \nabla
_{V}^{2}\phi \left( \Sigma \right) \Sigma +\Sigma \nabla _{V}^{2}\phi \left(
\Sigma \right) ,V\right\rangle _{\Sigma }+ \\ \operatorname{Tr}\left[ \nabla \phi \left(
\Sigma \right) V\mathcal{L}_{\Sigma }\left[V\right] \right] -\operatorname{Tr}\left[ \mathcal{L}_{\Sigma }\left[V\right] \nabla \phi \left(
\Sigma \right) \Sigma \mathcal{L}_{\Sigma }\left[V\right] \right] \ .
\end{multline*}

Plugging $V = \lyapunov \Sigma V \Sigma + \Sigma \lyapunov \Sigma V$ into the second term of the RHS, we have at last 
\begin{equation*}
\operatorname{Hess}\phi (\Sigma )\left( V,V\right) =\left\langle \nabla _{V}^{2}\phi
\left( \Sigma \right) \Sigma +\Sigma \nabla _{V}^{2}\phi \left( \Sigma
\right) ,V\right\rangle _{\Sigma }+\operatorname{Tr}\left[ \nabla \phi \left( \Sigma
\right) \mathcal{L}_{\Sigma }\left[V\right] \Sigma \mathcal{L}_{\Sigma
}\left[V\right] \right] \ .  \label{eq:FIN}
\end{equation*}

Relation \eqref{eq:FIN} substantiates the following important property that
links the Hessian to the derivative along a geodesic (see the proof of
Prop. 5.5.4 of \cite{absil|mahony|sepulchre:2008}).

\begin{proposition}
Let $\phi :$ $\operatorname{Sym}^{++}\left( n\right) \rightarrow {\reals}$ be a
smooth scalar field and define
\begin{equation*}
\varphi \left( t\right) =\phi \left( \exp _{\Sigma }\left( tV\right) \right)
\ .
\end{equation*}
It holds 
\begin{equation*}
\ddot{\varphi}\left( 0\right) =\operatorname{Hess}\phi (\Sigma )\left( V,V\right) \ .
\end{equation*}
\end{proposition}

\begin{proof}
By Proposition \ref{Prop:FFA} 
\begin{equation*}
\Sigma (t)=\operatorname{Exp}_{\Sigma }\left( tV\right) =\Sigma +tV+t^{2}\mathcal{L}
_{\Sigma }[V]\Sigma \mathcal{L}_{\Sigma }[V]
\end{equation*}
where $\Sigma (0)=\Sigma $ and $\dot{\Sigma}(0)=V.$ Hence $\dot{\varphi}
\left( t\right) =\left\langle \nabla \phi (\Sigma (t)),\dot{\Sigma}
(t)\right\rangle _{2}$, and 
\begin{equation*}
\ddot{\varphi}\left( t\right) =\left\langle \nabla ^{2}\phi (\Sigma (t))[
\dot{\Sigma}(t)],\dot{\Sigma}(t)\right\rangle _{2}+\left\langle \nabla \phi
(\Sigma (t)),\ddot{\Sigma}(t)\right\rangle _{2}\ 
\end{equation*}
that evaluated at $t=0$, provides 
\begin{equation*}
\ddot{\varphi}\left( 0\right) =\left\langle \nabla ^{2}\phi (\Sigma
)[V],V\right\rangle _{2}+2\left\langle \nabla \phi (\Sigma ),\mathcal{L}_{\Sigma }(V)\Sigma \mathcal{L}_{\Sigma }(V)\right\rangle _{2}.
\end{equation*}

In view of Eq.~\eqref{eq:FIN}, 
\begin{equation*}
\operatorname{Hess}\phi (\Sigma )\left( V,V\right) =\left\langle \nabla _{V}^{2}\phi
\left( \Sigma \right) ,V\right\rangle _{2}+2\left\langle \nabla \phi \left(
\Sigma \right) ,\mathcal{L}_{\Sigma }\left[V\right] \Sigma \mathcal{L}_{\Sigma }\left[V\right] \right\rangle =\ddot{\varphi}\left( 0\right) .
\end{equation*}
\end{proof}

\section{Conclusion}\label{sec:discussion}

In the present paper we have discussed in some detail the Wasserstein geometric properties of the Gaussian densities manifold. We have followed a known argument based on the geometric notion of submersion. We have improved upon what is known in the literature by offering  a number of further results. In particular, we have studied the geodesic surfaces and provided an explicit form for the Riemannian exponential. More important, a new formulation of the metric based on the field of operators $\Sigma \mapsto \lyapunov \Sigma \cdot$ is introduced. This field of operator expresses the Riemannian metric by the Frobenius inner product: $W_\Sigma(X,Y) = \scalarat 2 {\lyapunov \Sigma X}{Y}$. This gives rise to an explicit identification of the Riemannian gradient as well as  to the calculation of the Levi-Civita covariant derivative, through the partial derivatives of the metric. The equations of the parallel transport and of the  Riemannian Hessian have been also derived.

While the form of the natural gradient is simple and may be a source of applications such as those of interest in Machine Learning, the Levi-Civita covariant derivative turns out to be more involved and it is not clear how to use it in applications. However, we have produced a simpler form by the introduction of a special moving frame. In view of this issue, we have not proceeded in this paper to compute  other geometrical quantities of interest, like the curvature tensor.

Numerical as well as simulation methods for the relevant equation of the geometry, like geodesics, parallel transport, Hessians, should be also considered. Applications of special interest are in the area of the linear optimization, by means of the natural gradient as direction of increase and by using the Riemannian exponential as a retraction, cf. \cite{absil|mahony|sepulchre:2008} and in Amari monograph \cite{amari:2016}. Also, second order optimization methods (Newton method), via the Riemannian Hessian and the Riemannian exponential, cf. \cite{absil|mahony|sepulchre:2008} and \cite{amari:2016}, are source of promising researches.

The issue of a comparison between Fisher and Wasserstein metric is not discussed here as it is, for example, in Chevallier et al. \cite{chevallier|kalunga|angulo:2017}.

From the point of view of applications in Statistics and Machine Learning, the use of the full Gaussian model is not realistic in many cases. We expect our results to be used to compute the Wasserstein geometry induced on parsimonious sub-manifolds such as those listed below.

\begin{enumerate}
\item Sub-manifold of the correlation matrices i.e, with unitary diagonal elements. In this case, the tangent space at each point is the space of symmetric matrices with zero diagonal.
\item Sub-manifold of trace 1 matrices. This case is of particular interest in Physics and prompts for a generalization of the theory to complex Gaussians i.e., Gaussians densities on $\complex^n$. Such distributions have Hermitian covariant matrices, a case that is discussed in  \cite{bhatia|jain|lim:2018}.
\item Sub-manifold of the concentration matrices with a given sparsity pattern. Notice that concentration matrices and dispersion matrices are both elements of the same space $\ppsym n$. In this case the statistical interpretation of the Wasserstein distance is not available but nevertheless other interpretations of the distance are mentioned in the Introduction.
\end{enumerate}

%\begin{acknowledgements}
%\end{acknowledgements}

\bibliographystyle{spmpsci}      % mathematics and physical sciences
\bibliography{tutto}

\end{document}